\documentclass[12pt]{amsart}

\usepackage{amsmath,amsthm,amsfonts,amssymb,mathrsfs}%\usepackage{paralist,nicefrac}
\usepackage[all]{xy}

\RequirePackage[dvipsnames,usenames]{color}  %makes the standard dvips color names available

\usepackage[total={6.5in,8.75in}, top=1.2in, left=1.0in, includefoot]{geometry}
\usepackage{amssymb}
\usepackage{latexsym}
\usepackage{amsmath,amsthm}
\usepackage{aeguill}
\usepackage{amssymb}
\usepackage{mathrsfs}
\usepackage{hyperref}
\usepackage{etoolbox}
\usepackage{enumerate}
\usepackage[dvipsnames]{xcolor}
\usepackage{tikz}[2013/12/13]
\usepackage{tikzit}
\usepackage{tikz-cd}
\usepackage{bbold}
\tikzstyle{dot}=[fill=black, draw=black, shape=circle]

\tikzstyle{arrow}=[->]
\tikzstyle{axis}=[<->, draw=black]
\tikzstyle{axis red}=[draw=red, <->]
\tikzstyle{new edge style 0}=[-, draw={rgb,255: red,191; green,191; blue,191}]
\tikzstyle{arrow red}=[draw=red, ->]
\tikzstyle{line red}=[-, draw=red]

\makeatletter
\pretocmd{\chapter}{\addtocontents{toc}{\protect\addvspace{15\p@}}}{}{}
\pretocmd{\section}{\addtocontents{toc}{\protect\addvspace{3\p@}}}{}{}
\makeatother

\makeatletter
\def\@tocline#1#2#3#4#5#6#7{\relax
  \ifnum #1>\c@tocdepth % then omit
  \else
    \par \addpenalty\@secpenalty\addvspace{#2}%
    \begingroup \hyphenpenalty\@M
    \@ifempty{#4}{%
      \@tempdima\csname r@tocindent\number#1\endcsname\relax
    }{%
      \@tempdima#4\relax
    }%
    \parindent\z@ \leftskip#3\relax \advance\leftskip\@tempdima\relax
    \rightskip\@pnumwidth plus4em \parfillskip-\@pnumwidth
    #5\leavevmode\hskip-\@tempdima
      \ifcase #1
       \or\or \hskip .5em \or \hskip 1em \else \hskip 1.5em \fi%
      #6\nobreak\relax
    \dotfill\hbox to\@pnumwidth{\@tocpagenum{#7}}\par
    \nobreak
    \endgroup
  \fi}
\makeatother
\DeclareSymbolFont{bbold}{U}{bbold}{m}{n}
\DeclareSymbolFontAlphabet{\mathbbold}{bbold}

\newcommand{\C}{\mathbb{C}}
\newcommand{\N}{\mathbb{N}}
\newcommand{\Z}{\mathbb{Z}}
\newcommand{\R}{\mathbb{R}}
\newcommand{\Q}{\mathbb{Q}}
\newcommand{\F}{\mathbb{F}}
\newcommand{\A}{\mathbb{A}}

\newcommand{\X}{\mathbb{X}}

\newcommand{\Gal}{\operatorname{Gal}}

\renewcommand{\ss}{\operatorname{ss}}
\newcommand{\der}{\operatorname{der}}

\newcommand{\End}{\operatorname{End}}

\newcommand{\SO}{\mathrm{SO}}

\newcommand{\GL}{\mathrm{GL}}

\def\bG{\mathbf{G}}
\def\bH{\mathbf{H}}

\def\bT{\mathbf{T}}
\def\bC{\mathbf{C}}

\def\bU{\mathbf{U}}

\def\bpx{\begin{pmatrix}}
\def\epx{\end{pmatrix}}

\newcommand{\hto}{\hookrightarrow}

\newcommand\Lieg{\mathfrak{g}}

\newcommand\Lieq{\mathfrak{q}}

\newcommand\Liet{\mathfrak{t}}

\newtheorem{thm}{Theorem}[section]
\newtheorem{theorem}[thm]{Theorem}
\newtheorem{cor}[thm]{Corollary}

\newtheorem{prop}[thm]{Proposition}
\newtheorem{lemma}[thm]{Lemma}

\newtheorem{remark}[thm]{Remark}

\theoremstyle{definition}
\newtheorem{example}[thm]{Example}

\newtheorem{defi}[thm]{Definition}

\title[]{Rectangular representations and $\lambda$-independence of algebraic monodromy groups}

\author{Chun-Yin Hui}
\address{Department of
 Mathematics, The University of Hong
 Kong, Pokfulam, Hong Kong}
\email{chhui@maths.hku.hk, pslnfq@gmail.com}

\author{Wonwoong Lee}
\address{ Department of
 Mathematics Education, Chonnam National University, 77, Yongbong-ro, Buk-gu, Gwangju, Republic of Korea}
\email{dldnjsdnd041@gmail.com}


\thanks{}

\date{\today}

\begin{document}
\maketitle

\begin{abstract}
Let $\Lieg$ be a complex semisimple Lie algebra.
We define what it means for a finite dimensional representation of $\Lieg$
to be rectangular and completely classify faithful rectangular representations.
As an application, we obtain new $\lambda$-independence results 
on the algebraic monodromy groups of compatible systems of $\lambda$-adic Galois representations
of number fields.
\end{abstract}

{\tableofcontents}

\section{Introduction}
\subsection{Rectangular representations}\label{s1.1}
\subsubsection{}\label{s1.1.1} Let $\psi:\Lieg\to\End(V)$ be a finite dimensional representation of
a complex semisimple Lie algebra $\Lieg$. The following question
is of particular interest.\vspace{.08in}

\noindent \textbf{Q1}. If $\psi$ is faithful and $\Liet$ is a Cartan subalgebra of $\Lieg$, to 
what extent is $(\Lieg,\psi)$ determined by the restriction $\psi|_\Liet$ (i.e., the formal character)?
\vspace{.08in}

 When $\psi$ is irreducible, Larsen-Pink
 completely answered this question 
in the group-theoretic perspective \cite[Theorem 4]{LP90}.
Without assuming the irreducibility of $\psi$, it is proven 
that the number of $A_m$-factors of $\Lieg$ for $m\in\N\backslash\{1,2,3,4,5,7,8\}$ (resp. the parity of 
the number of $A_4$-factors of $\Lieg$) is determined by the formal character $\psi|_\Liet$ \cite[Theorems 2.14, 2.17]{Hui13},
where $A_m:=\mathfrak{sl}(m+1)$. Inspired by recent work on type $A_1$ automorphic Galois representations \cite{HL24},
we study a  refinement  of the above  question; see Q2 below.

For $d\in\Z_{\geq 0}$, define $Z_d:=\{-d,-d+2,-d+4,\ldots,d-2,d\}\subset\Z$.
The size $|Z_d|$ is $d+1$.
Let $n$ be the rank of $\Lieg$.
The weight lattice $\Lambda_\Lieg$ of $\Lieg$ (with respect to $\Liet$) 
is a $\Z$-lattice of the $n$-dimensional real vector space $\Lambda_\Lieg\otimes\R$. 
The formal character $\psi|_\Liet$
corresponds to a multiset 
$\Xi$ of weights in $\Lambda_\Lieg\subset \Lambda_\Lieg\otimes\R$. 
We say that the representation $\psi:\Lieg\to\End(V)$ is \emph{rectangular} (Definition \ref{rectrepn}) if
\begin{enumerate}
\item[(a)] every weight in $\Xi$ is of multiplicity one and
\item[(b)] there exist an $\R$-isomorphism $\iota:\Lambda_\Lieg\otimes\R\to\R^n$ and $d_1,\ldots,d_n\in\Z_{\geq 0}$ such that 
$$\iota(\Xi)=Z_{d_1}\times Z_{d_2}\times\cdots\times Z_{d_n}.$$
\end{enumerate}
Geometrically, (b) means that the weights $\Xi$ can be arranged in a rectangular shape in some coordinate system of $\Lambda_\Lieg\otimes\R$.
The \emph{set of lengths} of the rectangular representation $\psi$ is defined as the multiset 
$$\mathscr L_\psi:=\{|Z_{d_i}|=d_i+1:~ 1\leq i\leq n\}.$$
We say that the rectangular representation $\psi$ is \emph{hypercubic}  if $d_1=d_2=\cdots=d_n$.
In that case, %$\Xi$ can be made hypercubic in some coordinate system, and 
$d_1+1$ is called the \emph{length} of $\psi$.
All these terminology and notation are well-defined and depend only on 
the formal character $\psi|_\Liet$ (see $\mathsection\mathsection\ref{s2.2}-\ref{s2.3}$).

Typical examples of rectangular representations are irreducible representations 
of $A_1\times\cdots\times A_1$. If $(\Lieg_1,\psi_1)$ and $(\Lieg_2,\psi_2)$
are rectangular with the set of lengths, respectively, equal to $\mathscr L_{\psi_1}$ and $\mathscr L_{\psi_2}$, 
then the \emph{external tensor product} $(\Lieg_1\times\Lieg_2,\psi_1\otimes\psi_2)$
is rectangular with the set of lengths  given by the disjoint union $\mathscr L_{\psi_1}\sqcup \mathscr L_{\psi_2}$. 
A rectangular representation $(\Lieg,\psi)$ is said 
to be \emph{indecomposable} if it is not equivalent to an external tensor product
of two rectangular representations as above. Below is a natural question.\\

\noindent \textbf{Q2}. Classify all faithful rectangular representations of complex semisimple Lie algebras.

\subsubsection{}\label{s1.1.2} Adopt the conventions: $A_m=\mathfrak{sl}(m+1)$ ($m\geq 1)$, $B_m=\mathfrak{so}(2m+1)$ ($m\geq 1)$,
$C_m=\mathfrak{sp}(2m)$ ($m\geq 1)$, and $D_m=\mathfrak{so}(2m)$ ($m\geq 2)$ for classical complex 
Lie algebras. All of these are simple, except for $D_2=A_1\times A_1$. Note also the coincidence $A_1=B_1=C_1$, $B_2=C_2$, and $D_3=A_3$.
We present a complete answer to Q2.

\begin{thm}\label{main1}
Let $\psi$ be a faithful rectangular representation 
of a complex semisimple Lie algebra $\Lieg$. 
Fix a decomposition $\Lieg=\Lieg_1\times\Lieg_2\times\cdots\times \Lieg_k$, where
$\Lieg_1$ denotes the product of $A_1$-factors and $\Lieg_2,\ldots,\Lieg_k$ denote simple non-$A_1$-factors. 
Then the following assertions hold.
\begin{enumerate}
\item[(i)] There exist a faithful rectangular representation $(\Lieg_1,\psi_1)$
and faithful indecomposable hypercubic representations $(\Lieg_i,\psi_i)$ for $2\leq i\leq k$ such that 
\begin{equation}\label{uni1}
(\Lieg,\psi)=\bigotimes_{i=1}^k(\Lieg_i,\psi_i)
\end{equation}
as an external tensor product.
\item[(ii)] The rectangular representation $\psi_1$ admits an external tensor product 
of indecomposable hypercubic representations
\begin{equation}\label{uni2}
(\Lieg_1,\psi_1)=\bigotimes_{j=1}^{s} (\Lieg_{1,j},\psi_{1,j})
\end{equation}
such that $\Lieg_1=\prod_{j=1}^s\Lieg_{1,j}$ is some decomposition
and each $\psi_{1,j}$ is one of the following.
\begin{enumerate}[(a)]
    \item $(A_1,\mathrm{Sym}^r(\mathrm{Std}))$, $r \in \N$.
    \item $(A_1, \mathrm{Sym}^{r_1}(\mathrm{Std}) \oplus \mathrm{Sym}^{r_2}(\mathrm{Std}))$, $r_1,r_2\in\Z_{\geq 0}$ and $|r_1-r_2|=1$.
	% $r_1\in\N$ is odd and $r_2=r_1\pm 1$ is even.
    \item $(A_1 \times A_1, (\mathrm{Std} \otimes \mathbb 1) \oplus (\mathbb 1 \otimes \mathrm{Std}))
		=(D_2,\mathrm{Spin})$.
\end{enumerate}
\item[(iii)] The hypercubic representation $\psi_i$ for $2\leq i\leq k$ is one of the following.
\begin{enumerate}[(a)]
		\item $(B_2, \mathrm{Std}\oplus \mathrm{Spin})$.
    \item $(B_m, \mathrm{Spin})$, $m \geq 2$.
		\item $(A_3,\mathrm{Std}\oplus\mathrm{Std}^{\vee})=(D_3,\mathrm{Spin})$.
    \item $(D_4, \mathrm{Spin})$ or $(D_4, \mathrm{Std}\oplus \mathrm{Spin}^+)$ or $(D_4, \mathrm{Std}\oplus \mathrm{Spin}^-)$.
    \item $(D_m, \mathrm{Spin})$, $m \geq 5$.
\end{enumerate}
\item[(iv)] The external tensor product \eqref{uni1} is unique.
\item[(v)] The external tensor product \eqref{uni2} is unique up to permutation
of the factors in the decomposition $\Lieg_1=\prod_{j=1}^s\Lieg_{1,j}$.
\end{enumerate}
\end{thm}

\begin{remark}
Here, $\mathrm{Std}$ (resp. $\mathrm{Spin}$) denotes the standard (resp. spin) representation,
and $\mathbb 1$ denotes the trivial (one-dimensional) representation. 
For $m \geq 2$, $(B_m, \mathrm{Spin})$ is irreducible 
and $(D_m, \mathrm{Spin})$ is the direct sum of two irreducible (half-spin) representations $\mathrm{Spin}^+$ and $\mathrm{Spin}^-$.

Note that $(A_1,\mathrm{Std})\otimes (A_1,\mathrm{Std})$, $(B_2,\mathrm{Spin})=(C_2,\mathrm{Std})$, 
and $(D_2,\mathrm{Spin})$ are all hypercubic with the same set of lengths $\{2,2\}$.
In Theorem \ref{main1}(iii), the hypercubic $(B_2, \mathrm{Std}\oplus \mathrm{Spin})$ has length $3$, while the remaining hypercubic representations all have length $2$. 

Note also that the representation $\psi_{1,j}$ (resp. $\psi_i$) in Theorem \ref{main1}(ii) (resp. Theorem \ref{main1}(iii)) is either irreducible or the direct sum of two irreducibles.
We present the figures of the weights of the (reducible) representations in Theorem~\ref{main1}(ii)(b),(c) and 
Theorem~\ref{main1}(iii)(a),(c).
\end{remark}

\begin{figure}[h]
\begin{tabular}{cccccccc}
\begin{tikzpicture}
    \draw [->, red] (-5, 0) -- (5, 0);
    \node at (-4, 0) [circle, fill=black, inner sep=1.5pt] {};
    \node at (-2, 0) [circle, fill=black, inner sep=1.5pt] {};
    \node at (0, 0)  [circle, fill=black, inner sep=1.5pt] {};
    \node at (2, 0)  [circle, fill=black, inner sep=1.5pt] {};
    \node at (4, 0)  [circle, fill=black, inner sep=1.5pt] {};

    \node at (-3, 0) [circle, draw=black, fill=white, inner sep=1.5pt] {};
    \node at (-1, 0) [circle, draw=black, fill=white, inner sep=1.5pt] {};
    \node at (1, 0)  [circle, draw=black, fill=white, inner sep=1.5pt] {};
    \node at (3, 0)  [circle, draw=black, fill=white, inner sep=1.5pt] {};
\end{tikzpicture}
& & & & & & &
\begin{tikzpicture}
    \draw [->, red] (-2, 0) -- (2, 0);
    \draw [->, red] (0, -2) -- (0, 2);

    \node at (-1, 0) [circle, fill=black, inner sep=1.5pt] {};
    \node at (1, 0)  [circle, fill=black, inner sep=1.5pt] {};

    \node at (0, -1) [circle, draw=black, fill=white, inner sep=1.5pt] {};
    \node at (0, 1)  [circle, draw=black, fill=white, inner sep=1.5pt] {};
\end{tikzpicture}
\end{tabular}
\caption{$(A_1, \mathrm{Sym}^{4}(\mathrm{Std}) \oplus \mathrm{Sym}^{3}(\mathrm{Std}))$ and $(A_1 \times A_1, (\mathrm{Std} \otimes \mathbb 1) \oplus (\mathbb 1 \otimes \mathrm{Std}))$}
\end{figure}

\begin{figure}[h]
\begin{tabular}{cccccccc}
\begin{tikzpicture}[baseline={(0,0)}]
    \draw [->, red] (-3.5, 0) -- (3.5, 0);
    \draw [->, red] (0, -3.5) -- (0, 3.5);

    \node at (0, 0)  [circle, fill=black, inner sep=1.5pt] {};
    \node at (2, 0)  [circle, fill=black, inner sep=1.5pt] {};
    \node at (-2, 0) [circle, fill=black, inner sep=1.5pt] {};
    \node at (0, 2)  [circle, fill=black, inner sep=1.5pt] {};
    \node at (0, -2) [circle, fill=black, inner sep=1.5pt] {};

    \node at (1, 1)   [circle, draw=black, fill=white, inner sep=1.5pt] {};
    \node at (1, -1)  [circle, draw=black, fill=white, inner sep=1.5pt] {};
    \node at (-1, 1)  [circle, draw=black, fill=white, inner sep=1.5pt] {};
    \node at (-1, -1) [circle, draw=black, fill=white, inner sep=1.5pt] {};
\end{tikzpicture}
& & & & & & &
\begin{tikzpicture}[baseline={(0,0.23)}, scale=0.9]
    \draw [->, red] (-4.5, 0.5) -- (4.5, 0.5); 
    \draw [->, red] (0, -3.5) -- (0, 4.5);   
    \draw [->, red] (-2, -0.5) -- (2, 1.5);

    \draw [gray, dashed] (-2, 4) -- (-2, -2);   % 뒤쪽-왼쪽 세로선
    \draw [gray, dashed] (-4, -3) -- (-2, -2);  % 바닥-왼쪽 깊이선
    \draw [gray, dashed] (4, -2) -- (-2, -2);   % 바닥-뒤쪽 가로선

    % 보이는 모서리 (실선 처리, 앞쪽에 그려짐)
    \draw [gray] (-4, 3) -- (2, 3) -- (4, 4) -- (-2, 4) -- cycle; % 윗면
    \draw [gray] (-4, -3) -- (2, -3) -- (4, -2);                  % 아랫면 (보이는 테두리)
    \draw [gray] (-4, 3) -- (-4, -3); % 앞쪽-왼쪽 세로선
    \draw [gray] (2, 3) -- (2, -3);   % 앞쪽-오른쪽 세로선
    \draw [gray] (4, 4) -- (4, -2);   % 뒤쪽-오른쪽 세로선

    \node at (-2, 4)  [circle, fill=black, inner sep=1.5pt] {};
    \node at (2, 3)   [circle, fill=black, inner sep=1.5pt] {};
    \node at (-4, -3) [circle, fill=black, inner sep=1.5pt] {};
    \node at (4, -2)  [circle, fill=black, inner sep=1.5pt] {};

    \node at (-4, 3)  [circle, draw=black, fill=white, inner sep=1.5pt] {};
    \node at (4, 4)   [circle, draw=black, fill=white, inner sep=1.5pt] {};
    \node at (2, -3)  [circle, draw=black, fill=white, inner sep=1.5pt] {};
    \node at (-2, -2) [circle, draw=black, fill=white, inner sep=1.5pt] {};

\end{tikzpicture}
\end{tabular}
\caption{$(B_2, \mathrm{Std}\oplus \mathrm{Spin})$ and $(A_3,\mathrm{Std}\oplus\mathrm{Std}^{\vee})$}
\end{figure}

The following corollaries are immediate.

\begin{cor}
If $\psi$ is a faithful rectangular representation of a complex semisimple Lie algebra $\Lieg$,
then each simple factor of $\Lieg$ 
must be one of $A_1$, $A_3$, $B_m$ for $m\geq 2$, or $D_m$ for $m\geq 4$.
\end{cor}

\begin{cor}\label{cor_L at most 2 at most 3}
Let $\psi$ be a faithful rectangular representation of a complex semisimple Lie algebra $\Lieg$ with $\mathscr L_\psi$ the set of lengths.
\begin{enumerate}[(i)]
\item If $\mathscr L_\psi$ contains at most one $2$ and at most one $3$, then $\Lieg$
has only $A_1$-factors.
\item If every element of $\mathscr L_\psi$ is even, and 2 appears at most once in $\mathscr L_\psi$, then the representation decomposes as an external tensor product $$(\Lieg, \psi)=\bigotimes_{\ell\in\mathscr L_\psi} (A_1,\mathrm{Sym}^{\ell-1}(\mathrm{Std}))$$ and it is irreducible.
\end{enumerate}
\end{cor}

\begin{cor}\label{irrednumber}
If $\psi$ is a rectangular representation of a complex semisimple Lie algebra $\Lieg$,
then $\psi$ is the direct sum of $2^t$ irreducible representations for some integer $t$.
\end{cor}

\begin{defi}\label{rectforgp}
Let $\Psi:\bG\to\GL(V)$ be a finite dimensional complex representation of a complex reductive group $\bG$
and let $\psi: \mathrm{Lie}(\bG)\to\End(V)$ be the associated Lie algebra representation.
We say that $\Psi$ is rectangular (resp. hypercubic) if this is so for the restriction $\psi|_{\mathrm{Lie}(\bG)^{\ss}}$,
where $\mathrm{Lie}(\bG)^{\ss}$ is the semisimple part of $\mathrm{Lie}(\bG)$.
Then the set of lengths of $\Psi$ is defined to be the set of lengths of $\psi|_{\mathrm{Lie}(\bG)^{\ss}}$.
\end{defi}

The classification results above are applied to deduce new  $\lambda$-independence results  on the
algebraic monodromy groups $\{\bG_\lambda\}$ of some compatible system $\{\rho_\lambda\}$
of Galois representations.

\subsection{Application to Galois representations}\label{s1.2}
\subsubsection{}\label{s1.2.1} Let $K$ be a number field, $\overline K$ be an algebraic closure of $K$, 
and $\Gal_K:=\Gal(\overline K/K)$ the absolute Galois group of $K$ (equipped with the profinite topology).
Let $E$ be another number field and $\Sigma_E$ be the set of finite places of $E$. 
For each finite place $\lambda\in\Sigma_E$, denote by $E_\lambda$ the $\lambda$-adic completion of $E$ and by $\overline E_\lambda$
an algebraic closure of $E_\lambda$.
A $\lambda$-adic representation of $K$ is a continuous group homomorphism
$$\rho_\lambda:\Gal_K\to\GL_n(\overline E_\lambda).$$
The \emph{algebraic monodromy group} of $\rho_\lambda$, denoted $\bG_\lambda$, is the Zariski closure of the Galois image $\rho_\lambda(\Gal_K)$ in $\GL_{n,\overline E_\lambda}$.
If $\rho_\lambda(\Gal_K)\subset\GL_n(E_\lambda)$, then $\bG_\lambda$
becomes an $E_\lambda$-subgroup of $\GL_{n,E_\lambda}$. 
If $\rho_\lambda$ is semisimple,
then $\bG_\lambda$ is reductive and $\bG_\lambda^{\der}:=[\bG_\lambda^\circ,\bG_\lambda^\circ]$ 
(the derived group of the identity component $\bG_\lambda^\circ$) is semisimple.
Fix an isomorphism $\iota_\lambda:\overline E_\lambda\simeq \C$ for each $\lambda\in\Sigma_E$.
We define $\bG_{\lambda,\C}:=\bG_\lambda\times_{\iota_\lambda}\C$
and $\bG_{\lambda,\C}^{\der}:=\bG_\lambda^{\der}\times_{\iota_\lambda}\C$.

Given a \emph{compatible system} (Definition \ref{csdef}) of semisimple $\lambda$-adic Galois representations
\begin{equation}\label{cs}
\{\rho_\lambda:\Gal_K\to\GL_n(\overline E_\lambda)\}_{\lambda\in\Sigma_E}
\end{equation}
 of $K$ defined over $E$, we are interested in the following question on the system of 
algebraic monodromy groups $\{\bG_\lambda\}_{\lambda\in\Sigma_E}$. \\

\noindent \textbf{Q3}. ($\lambda$-independence of $\bG_\lambda\subset\GL_{n,\overline E_\lambda}$)
\begin{enumerate}[(i)]
\item Is the conjugacy class of the complex reductive subgroup
$\bG_{\lambda,\C}\subset \GL_{n,\C}$
 independent of $\lambda\in\Sigma_E$?
\item Suppose $\rho_\lambda(\Gal_K)\subset\GL_n(E_\lambda)$ for each $\lambda\in\Sigma_E$.
Does there exist a reductive subgroup $\bG\subset\GL_{n,E}$ such that 
$\bG\times_E E_\lambda$ and $\bG_\lambda$ are conjugate in $\GL_{n,E_\lambda}$
for all $\lambda\in\Sigma_E$?
\end{enumerate}

When $\rho_\lambda$ is abelian for some $\lambda$, the answer to Q3(ii) is yes \cite{Ser98}.
When the compatible system is associated to the $\ell$-adic \'etale cohomology groups of degree $w$ of
a smooth projective variety $X/K$ (in this case $E=\Q$), the Mumford-Tate conjecture (see e.g., \cite[$\mathsection9$]{Ser94}) predicts 
a candidate for the identity component $\bG^\circ$ in Q3(ii), see e.g. \cite{LP95},\cite{Pi98} 
for positive results on some abelian varieties.
When $E=\Q$, Larsen-Pink  obtained various $\ell$-independence results of $\bG_\ell^\circ$ 
after restricting to some Dirichlet
density one set of rational primes $\ell\in\Sigma_\Q$ \cite{LP92}; the first author
studied Q3 under some type $A$ assumption on one $\bG_\ell$ \cite{Hui18}.
Assuming that Q3(i) holds, some criteria are given for constructing $\bG^\circ$ in Q3(ii)
in \cite{Hui25}.
If the compatible system is associated to some automorphic representation of a totally real field $K$,
Q3(i) is studied under some group-theoretic assumptions  on some $\bG_\lambda$ (e.g. connected and of type $A_1$) \cite{HL24}.

Although Q3 is wide open, the following $\lambda$-independence results are known (see $\mathsection\ref{s3.5}$ for details).

\begin{enumerate}[(A)]
\item The component group $\bG_\lambda/\bG_\lambda^\circ$ is independent of $\lambda$ \cite{Ser81}.
\item The \emph{formal character} (Definition~\ref{formchar}(1)) of the reductive subgroup $\bG_{\lambda,\C}\subset\GL_{n,\C}$ is independent 
of $\lambda$  \cite{Ser81}. %\footnote{If $\bT_{\lambda,\C}$ is a maximal torus of $\bG_{\lambda,\C}$, then the conjugacy class of $\bT_{\lambda,\C}$ in $\GL_{n,\C}$ is independent of $\lambda$.} 
\item The \emph{formal bi-character} (Definition~\ref{formchar}(2)) of the reductive subgroup $\bG_{\lambda,\C}\subset\GL_{n,\C}$ 
is independent of $\lambda$  \cite{Hui13}. %\footnote{If $\bT_{\lambda,\C}^{\ss}\subset\bT_{\lambda,\C}$ is a chain of tori such that $\bT_{\lambda,\C}$ is a maximal torus of $\bG_{\lambda,\C}$ and $\bT_{\lambda,\C}^{\ss}$ is a maximal torus of $\bG_{\lambda,\C}^{\der}$, then the conjugacy class of the chain in $\GL_{n,\C}$ is independent of $\lambda$.} 
In particular, the formal character of the semisimple subgroup $\bG_{\lambda,\C}^{\der}\subset\GL_{n,\C}$ is independent of $\lambda$.
\end{enumerate}
We say that the compatible system \eqref{cs} is \emph{connected} if $\bG_\lambda$
is connected for some $\lambda$ (hence for all $\lambda$, by assertion (A)).
Via assertion (C), we can connect the $\lambda$-independence of $\bG_\lambda\hookrightarrow\GL_{n,\overline E_\lambda}$ with the terminology  
in $\mathsection\ref{s1.1}$ on rectangular representations (that depend solely on the formal character of the semisimple part)
as follows.

\begin{enumerate}[(D)]
\item If the faithful representation $\bG_{\lambda,\C}\hookrightarrow\GL_{n,\C}$ is rectangular (Definition \ref{rectforgp}) for one $\lambda$,
this is true for all $\lambda$. Moreover, the set of lengths of the rectangular $\bG_{\lambda,\C}\hookrightarrow\GL_{n,\C}$
is independent of $\lambda$.
\end{enumerate}

\noindent By (D), it makes sense to call the semisimple compatible system $\{\rho_\lambda\}$ \emph{rectangular} (resp. \emph{hypercubic}) if the %semisimple part of the Lie algebra of 
faithful representation $\bG_{\lambda,\C}\hto\GL_{n,\C}$ is rectangular (resp. hypercubic)  for some $\lambda$;
the \emph{set of lengths} $\mathscr L_\rho$  of the rectangular  compatible system $\{\rho_\lambda\}$ is defined to be 
the set of lengths  of $\bG_{\lambda,\C}\hookrightarrow\GL_{n,\C}$.

\subsubsection{}\label{s1.2.2} We now state our results on $\lambda$-independence of algebraic monodromy groups
for certain compatible systems, relating them to the classification results in $\mathsection\ref{s1.1.2}$.
First, as a direct consequence of Corollary \ref{cor_L at most 2 at most 3}(ii) and (C) above, we obtain a positive answer 
(on the identity components $\bG_{\lambda,\C}^\circ$) to question Q3(i) 
for certain rectangular compatible systems.

\begin{thm}\label{thm_main3}
    Let $K$ be a number field and $\{\rho_{\lambda}: \mathrm{Gal}_K \to \mathrm{GL}_n(\overline E_{\lambda})\}_{\lambda \in \Sigma_E}$ 
		be a compatible system of semisimple $\lambda$-adic representations of $K$ defined over a number field $E$. 
		Suppose the following conditions are satisfied.
    \begin{enumerate}[(a)]
				\item $\{\rho_{\lambda}\}$ is rectangular.
				\item The set of lengths $\mathscr L_\rho$ has only even numbers and at most one $2$.
    \end{enumerate}
		Then the algebraic monodromy group $\bG_{\lambda,\C}$ is of type $A_1$ (i.e., $\mathrm{Lie}(\bG_{\lambda,\C}^{\der})=A_1\times\cdots\times A_1$), 
		the identity component $\bG_{\lambda,\C}^\circ$ is irreducible on the ambient space 
		and its conjugacy class in $\mathrm{GL}_{n,\C}$ is independent of $\lambda$. 
		
    %Then $\bG_{\lambda}$ is of type $A_1$ for all $\lambda$ and the following statements hold.
   % \begin{enumerate}
     %   \item[\normalfont(i)] The conjugacy class of $\bG_{\lambda,\C}$ in $\mathrm{GL}_{n,\C}$ is independent of $\lambda$ under identification $\overline E_{\lambda} \simeq \mathbb C$. 
      %  \item[\normalfont(ii)] The representation $\rho_{\pi, \lambda}$ is irreducible for all $\lambda$.
   % \end{enumerate}
  %  If we further assume that $\rho(\mathrm{Gal}_K) \subset \mathrm{GL}_n(E_{\lambda})$ for all $\lambda$, then the following statement also holds.
   % \begin{enumerate}
	%\item[\normalfont(iii)] The representation $\rho_{\pi, \lambda}$ is residually irreducible for almost all $\lambda$.
 %  \end{enumerate}
\end{thm}

% %version 1
% \begin{thm}\label{thm_main2}
%     Let $K$ be a totally real field, $E$ be a number field, and let $\{\rho_{\lambda}: \mathrm{Gal}_K \to \mathrm{GL}_n(E_{\lambda})\}_{\lambda \in \sum_E}$ be a strictly compatible system of semisimple Galois representations of $K$ defined over $E$. Suppose there exists $\lambda_0 \in \sum_E$ such that
%     \begin{enumerate}
%         \item[\normalfont(a)] The algebraic monodromy group $\bG_{\lambda_0} \times_{E_{\lambda_0}} \overline E_{\lambda_0}$ is connected and its formal character is rectangular.
%         \item[\normalfont(b)] Both the number of basic factors $(\mathrm{SL}_{2},\mathrm{std})$ and the number of basic factors $(\mathrm{SL}_{2},\mathrm{Sym}^2)$ in the exterior tensor decomposition of the tautological representation of $\bG_{\lambda_0}^{\mathrm{der}} \times_{E_{\lambda_0}} \overline E_{\lambda_0}$ are at most $1$.
%     \end{enumerate}
%     If we assume either
%     \begin{enumerate}
%         \item[\normalfont(c1)] A compatible system $\{\rho_{\lambda}\}_{\lambda}$ is regular, or
%         \item[\normalfont(c2)] $K=\mathbb Q$ and every two-dimensional irreducible subrepresentation of the Hodge-Tate lift of $\rho_{\lambda_0}$ is odd, 
%     \end{enumerate}
%     then the conjugacy class of $\bG_{\lambda} \times_{E_{\lambda}}\mathbb C$ in $\mathrm{GL}_{n,\C}$ is independent of $\lambda$ under identification $\overline E_{\lambda} \simeq \mathbb C$.
% \end{thm}

%version 2
Now assume $K$ is totally real. % and $\rho_\lambda$ is $\GL_n(E_\lambda)$-valued for all $\lambda$.
By adopting the techniques in \cite{Hui23a,Hui23b},\cite{HL24} (combining various Galois theoretic
and representation theoretic results), we obtain the following $\lambda$-independence results.

\begin{thm}\label{thm_main2}
    Let $K$ be a totally real field and $\{\rho_{\lambda}: \mathrm{Gal}_K \to \mathrm{GL}_n(\overline{E}_{\lambda})\}_{\lambda \in \Sigma_E}$ be a compatible system of semisimple $\lambda$-adic representations of $K$ defined over a number field $E$. 
		Suppose the following conditions are satisfied.
    \begin{enumerate}[(a)]
		\item $\{\rho_{\lambda}\}$ is strictly compatible and regular.
        \item $\{\rho_{\lambda}\}$ is connected.
        \item $\{\rho_{\lambda}\}$ is rectangular.
				\item  The set of lengths $\mathscr L_\rho$ has at most one $2$ and at most one $3$.
  %  If we assume either
   % \begin{enumerate}
     %   \item[\normalfont(c1)] A compatible system $\{\rho_{\lambda}\}_{\lambda}$ is regular, or
     %   \item[\normalfont(c2)] $K=\mathbb Q$ and every two-dimensional irreducible subrepresentation of the $\mathrm{GL}_2^r$-valued lift (Definition~\ref{def_GL2r lift}) of an irreducible subrepresentation of $\rho_{\lambda_0}$ is odd, 
    \end{enumerate}
    Then the algebraic monodromy group $\bG_{\lambda,\C}$ is of type $A_1$  
		and its conjugacy class in $\mathrm{GL}_{n,\C}$ is independent of $\lambda$.
\end{thm}

\begin{remark}\label{remark1.7}
Here are a few remarks about the conditions of the theorem.
Let $\pi$ be a regular algebraic, polarized, cuspidal automorphic representation of $\GL_n(\A_K)$,
where $K$ is totally real.
If the compatible system is associated to $\pi$ (see e.g., \cite{BLGGT14}),
then (a) holds. If, for some $\lambda$, the semisimple group $\bG_{\lambda,\C}^{\der}$
is of type $A_1$ and acts irreducibly on the ambient space, then (c) holds.
If (c) holds and the dimension $n$ is neither divisible by $4$ nor by $9$, then obviously (d) holds.
Note that (b),(c),(d) and Corollary \ref{cor_L at most 2 at most 3}(i) imply that $\bG_{\lambda,\C}$ is connected and 
of type $A_1$ for all $\lambda$. Note that the representation $\bG_{\lambda,\C}\hto \mathrm{GL}_{n,\C}$ can be reducible.

\end{remark} 

\begin{cor}\label{cor1.8}
    Let $K$ be a totally real field and $\{\rho_{\lambda}: \mathrm{Gal}_K \to \mathrm{GL}_n(\overline{E}_{\lambda})\}_{\lambda \in \Sigma_E}$ be a compatible system of semisimple $\lambda$-adic representations of $K$ defined over a number field $E$. 
		Suppose the following conditions are satisfied.
    \begin{enumerate}[(a)]
        \item $\{\rho_{\lambda}\}$ is strictly compatible and regular.
				%The representation $\rho_{\pi, \lambda_0}$ is irreducible. 
        \item $\{\rho_{\lambda}\}$ is connected.
				%The algebraic monodromy group $\bG_{\lambda_0,\overline E_{\lambda_0}}$ is connected and of type $A_1$.
        \item For some $\lambda_0$, $\bG_{\lambda_0,\C}^{\der}$ is of type $A_1$ and is irreducible on the ambient space.
				\item The set of lengths $\mathscr L_\rho$ has at most one $2$ and at most one $3$.
				%Both the number of basic factors $(\mathrm{SL}_{2},\mathrm{std})$ and the number of basic factors $(\mathrm{SL}_{2},\mathrm{Sym}^2)$ in the \textcolor[rgb]{1,0,0}{external} tensor decomposition of the tautological representation of $\bG_{\lambda_0,\overline E_{\lambda_0}}^{\mathrm{der}}$ are at most $1$.
    \end{enumerate}
    Then the following statements hold.
    \begin{enumerate}[(i)]
        \item The conjugacy class of $\bG_{\lambda,\C}$ in $\mathrm{GL}_{n,\C}$ is independent of $\lambda$.
        \item The representation $\rho_{\lambda}$ is absolutely irreducible for all $\lambda$.
		\item The residual representation of $\rho_{\lambda}$ is absolutely irreducible for almost all $\lambda$.
   \end{enumerate}
\end{cor}

\begin{remark}
When we replace \ref{cor1.8}(a) by the stronger condition that $\{\rho_\lambda\}$
is associated to the automorphic representation $\pi$ in Remark \ref{remark1.7}, 
this result is obtained previously in \cite[Theorem~1.1]{HL24}.
Therefore, Corollary \ref{cor1.8} can be viewed as 
a partial generalization of \cite[Theorem~1.1]{HL24} from automorphic compatible systems 
to regular strictly compatible systems.
\end{remark}

\subsection{Organization of the article}\label{s1.3}
We describe the remaining sections of the article.

Section $2$ is devoted to the proof of the classification theorem of rectangular representations (Theorem \ref{main1}).
It is purely representation theoretic and can be read independently.
We first present the basics (e.g., the root system and Weyl group) of the simple Lie algebra $B_n$ in $\mathsection\ref{s2.1}$ and 
define rectangular representations in $\mathsection\mathsection\ref{s2.2}-\ref{s2.3}$.
We perform certain reductions of the classification theorem in $\mathsection\mathsection\ref{s2.4}-\ref{s2.6}$,
 where $\mathsection\ref{s2.1}$ and a result of Wright on minimal embeddings of symmetric groups \cite{Wri75}
are critical. Then Theorem \ref{main1}(ii) is proven in $\mathsection\ref{s2.7}$.
By using the classification 
of irreducible multiplicity-free representations of simple Lie algebra due to Howe \cite{How92},
Theorem \ref{main1}(iii) is obtained in $\mathsection\ref{s2.8}$ via a case-by-case analysis.
Finally, we prove Theorem \ref{main1}(iv),(v) (the uniqueness part) in $\mathsection\ref{s2.9}$.

Section $3$ is devoted to the proof of Theorem \ref{thm_main2} (and Corollary \ref{cor1.8}), 
concerning the $\lambda$-independence of algebraic monodromy groups $\bG_{\lambda,\C}\subset\GL_{n,\C}$
of a strictly compatible system. We briefly describe the strategy.
Building on previous works \cite{Hui23a,Hui23b},\cite{HL24}, 
we establish Proposition \ref{strategy} which asserts that
the $\lambda$-independence of $\bG_{\lambda,\C}\subset\GL_{n,\C}$ holds
if some auxiliary compatible system $\{\phi_\lambda\}$,
derived from the adjoint representations on $\mathrm{Lie}(\bG_{\lambda}^{\der})$, exists.
The proof of Proposition \ref{strategy} hinges on 
two key components:
\begin{itemize}
\item a refinement (Theorem \ref{refine}) of the $\lambda$-independence of the formal (bi-)character of a compatible system
to the direct sum of multiple compatible systems, such as $\{\rho_\lambda\oplus\phi_\lambda\}$, and
\item  an invariance of roots criterion (Proposition \ref{invroot}) that guarantees
two connected reductive subgroups of $\GL_{n,\C}$ sharing
the same maximal torus
are conjugate in $\GL_{n,\C}$.
\end{itemize}

\noindent To obtain  Theorem \ref{thm_main2}, it suffices
to construct the auxiliary compatible system $\{\phi_\lambda\}$ 
as specified in Proposition \ref{strategy}.
Under the conditions \ref{thm_main2}(c),(d)
(which require the compatible system to be rectangular with a certain set of lengths), one deduces by Corollary \ref{cor_L at most 2 at most 3}(i) that
$\bG_{\lambda,\C}$ is of type $A_1$ for all $\lambda$.
This result, together with the conditions \ref{thm_main2}(a),(b), enable us to construct $\{\phi_\lambda\}$
by using various Galois lifting, big image, potential automorphy, and $\lambda$-independence results.
Roughly speaking, the type $A_1$ assertion on $\bG_{\lambda,\C}$ ensures
the three-dimensional
factors of the Galois representation on $\mathrm{Lie}(\bG_{\lambda_0}^{\der})$ (for some $\lambda_0$) satisfy the ``odd essential self-duality'' and ``irreducibility'' conditions in the potential automorphy theorem (Theorem \ref{thm_BLGGT14 Thm C}), which is used to extend $\mathrm{Lie}(\bG_{\lambda_0}^{\der})$ to the compatible system $\{\phi_\lambda\}$.

We present basic notation for Galois representations in $\mathsection\ref{s3.1}$ 
and define (strictly/automorphic) compatible system in $\mathsection\ref{s3.2}$.
The required techniques are explained separately throughout $\mathsection\mathsection\ref{s3.3}-\ref{s3.8}$.
By adopting the strategy above, we obtain Theorem \ref{thm_main2} and Corollary \ref{cor1.8} 
 finally in $\mathsection\mathsection\ref{s3.9}-\ref{s3.10}$.

\section{Classification of rectangular representations}
\subsection{Root system and Weyl group of $B_n$}\label{s2.1}
\begin{defi}\label{std}
For $n\in\N$, equip $\R^n$ with the standard Euclidean metric and let $\mathcal B:=\{e_1,e_2,...,e_n\}$ be 
the standard orthonormal basis of $\R^n$. A subspace of $\R^n$ is said to be \emph{standard}
if it is spanned by some elements in $\mathcal B$.
\end{defi}

 The following lemma is obvious.

\begin{lemma}\label{stdsum}
A subspace of $\R^n$ generated by some standard subspaces is standard.
\end{lemma}

We describe the root systems $\Phi$ of the classical Lie algebras 
$B_n=\mathfrak{so}(2n+1)$ ($n\geq 1$), $C_n=\mathfrak{sp}(2n)$ (for $n\geq 1$), 
and $D_n=\mathfrak{so}(2n)$ (for $n\geq 2$) as subset of the Euclidean space 
$\R^n$ as follows:
\begin{align}\label{rootsys}
\begin{split}
\Phi_{B_n}:&=\{\pm e_i, \pm e_i\pm e_j\in\R^n:~ 1\leq i< j\leq n\},\\
\Phi_{C_n}:&=\{\pm 2e_i, \pm e_i\pm e_j\in\R^n:~ 1\leq i< j\leq n\},\\
\Phi_{D_n}:&=\{\pm e_i\pm e_j\in\R^n:~ 1\leq i< j\leq n\}.
\end{split}
\end{align}

\vspace{.1in}
\noindent We present the figures of $\Phi_{B_2}$, $\Phi_{B_3}$ and $\Phi_{D_3}=\Phi_{A_3}$
and establish some useful facts about $\Phi_{B_n}$.

\begin{figure}[h]
\begin{tabular}{cccccccc}
\begin{tikzpicture}
	\begin{pgfonlayer}{nodelayer}
		\node [style=none] (0) at (-2, 0) {};
		\node [style=none] (1) at (2, 0) {};
		\node [style=none] (2) at (0, -2) {};
		\node [style=none] (3) at (0, 2) {};
		\node [style=none] (4) at (-2, -2) {};
		\node [style=none] (5) at (2, 2) {};
		\node [style=none] (6) at (-2, 2) {};
		\node [style=none] (7) at (2, -2) {};
		\node [style=none] (9) at (0, 2.5) {$e_2$};
		\node [style=none] (10) at (-2, 2) {};
		\node [style=none] (11) at (2, 2) {};
		\node [style=none] (12) at (-2, -2) {};
		\node [style=none] (13) at (2, -2) {};
		\node [style=none] (14) at (2.5, 0) {$e_1$};
        % \node at (2,2) [circle,fill,inner sep=1.5pt]{};
        % \node at (-2,2) [circle,fill,inner sep=1.5pt]{};
        % \node at (2,-2) [circle,fill,inner sep=1.5pt]{};
        % \node at (-2,-2) [circle,fill,inner sep=1.5pt]{};
	\end{pgfonlayer}
	\begin{pgfonlayer}{edgelayer}
		\draw [style=axis red] (0.center) to (1.center);
		\draw [style=axis red] (2.center) to (3.center);
		\draw [style=axis] (4.center) to (5.center);
		\draw [style=axis] (7.center) to (6.center);
	\end{pgfonlayer}
\end{tikzpicture}
& & & & & & &
\begin{tikzpicture}
	\begin{pgfonlayer}{nodelayer}
		\node [style=none] (0) at (-3, 0) {};
		\node [style=none] (1) at (3, 0) {};
		\node [style=none] (2) at (0, 3) {};
		\node [style=none] (3) at (0, -3) {};
		\node [style=none] (4) at (-1.3, -1) {};
		\node [style=none] (5) at (1.2, 1) {};
		\node [style=none] (6) at (-1, 2) {};
		\node [style=none] (7) at (1, 4) {};
		\node [style=none] (8) at (-1.3, -4) {};
		\node [style=none] (9) at (1, -2) {};
		\node [style=none] (10) at (-1.3, 2) {};
		\node [style=none] (11) at (0, 0) {};
		\node [style=none] (12) at (-4.2, -1) {};
		\node [style=none] (13) at (1.8, -1) {};
		\node [style=none] (14) at (4, 1) {};
		\node [style=none] (15) at (-2, 1) {};
		\node [style=none] (16) at (-4.2, 2) {};
		\node [style=none] (17) at (-4.2, -4) {};
		\node [style=none] (18) at (1.8, -4) {};
		\node [style=none] (19) at (1.8, 2) {};
		\node [style=none] (20) at (-2, 4) {};
		\node [style=none] (21) at (4, 4) {};
		\node [style=none] (22) at (-2, -2) {};
		\node [style=none] (23) at (4, -2) {};
		\node [style=none] (24) at (-3, 3) {};
		\node [style=none] (25) at (-3, -3) {};
		\node [style=none] (26) at (3, 3) {};
		\node [style=none] (27) at (3, -3) {};
		\node [style=none] (28) at (3.6, 0) {$e_2$};
		\node [style=none] (29) at (-0.9, -1.6) {$e_1$};
		\node [style=none] (30) at (0, 3.6) {$e_3$};
        % \node at (-4.2,-1) [circle,fill,inner sep=1.5pt]{};
        % \node at (1.8,-1) [circle,fill,inner sep=1.5pt]{};
        % \node at (4,1) [circle,fill,inner sep=1.5pt]{};
        % \node at (-2,1) [circle,fill,inner sep=1.5pt]{};
        % \node at (-1,2) [circle,fill,inner sep=1.5pt]{};
        % \node at (-1,-4) [circle,fill,inner sep=1.5pt]{};
        % \node at (1,-2) [circle,fill,inner sep=1.5pt]{};
        % \node at (1,4) [circle,fill,inner sep=1.5pt]{};
        % \node at (3,3) [circle,fill,inner sep=1.5pt]{};
        % \node at (-3,3) [circle,fill,inner sep=1.5pt]{};
        % \node at (3,-3) [circle,fill,inner sep=1.5pt]{};
        % \node at (-3,-3) [circle,fill,inner sep=1.5pt]{};
	\end{pgfonlayer}
	\begin{pgfonlayer}{edgelayer}
		\draw (16.center) to (19.center);
		\draw (16.center) to (17.center);
		\draw (17.center) to (18.center);
		\draw (18.center) to (19.center);
		\draw (16.center) to (20.center);
		\draw (20.center) to (21.center);
		\draw (19.center) to (21.center);
		\draw (18.center) to (23.center);
		\draw (23.center) to (21.center);
		\draw [style=new edge style 0] (20.center) to (22.center);
		\draw [style=new edge style 0] (22.center) to (17.center);
		\draw [style=new edge style 0] (22.center) to (23.center);
		\draw [style=axis red] (0.center) to (1.center);
		\draw [style=axis red] (3.center) to (2.center);
		\draw [style=arrow] (11.center) to (12.center);
		\draw [style=arrow] (11.center) to (15.center);
		\draw [style=arrow] (11.center) to (13.center);
		\draw [style=arrow] (11.center) to (14.center);
		\draw [style=arrow] (11.center) to (8.center);
		\draw [style=arrow] (11.center) to (7.center);
		\draw [style=arrow] (11.center) to (10.center);
		\draw [style=arrow] (11.center) to (24.center);
		\draw [style=arrow] (11.center) to (27.center);
		\draw [style=arrow] (11.center) to (9.center);
		\draw [style=arrow] (11.center) to (25.center);
		\draw [style=arrow red] (11.center) to (4.center);
		\draw [style=arrow red] (11.center) to (5.center);
		\draw [style=arrow] (11.center) to (26.center);
	\end{pgfonlayer}
\end{tikzpicture}
\end{tabular}
\caption{$\Phi_{B_2}$ (left) and $\Phi_{B_3}$ (right, and $\Phi_{D_3}$ consists of long roots in $\Phi_{B_3}$)}
\end{figure}

\begin{prop}\label{Bnroot} 
The following assertions hold about the root system $\Phi_{B_n}$ of $B_n$.
\begin{enumerate}[(i)]
\item Any root $\alpha\in\Phi_{B_n}$ in \eqref{rootsys} has at most two non-zero coordinates.
\item If $\alpha\in \Phi_{B_n}$, then $-\alpha\in \Phi_{B_n}$ and the line spanned by $\alpha$ has exactly two roots. 
\item Suppose $n\geq 2$, and let $\alpha=\epsilon_1e_i+\epsilon_2e_j$ and $\beta=\epsilon_3 e_k+\epsilon_4 e_l$ 
be two linearly independent long roots in $\Phi_{B_n}$ (where 
$\epsilon_1,\epsilon_2,\epsilon_3,\epsilon_4\in\{\pm1\}$), 
$P$ be the plane spanned by $\{\alpha,\beta\}$,
and $d$ be the dimension of the space spanned by $\{e_i,e_j,e_k,e_l\}$. If $d=4$ (resp. $d=3$),
then $P$ has exactly $4$ (resp. $6$) roots and all of them are long.
\item If $n\geq 2$ and $P$ is a plane in $\R^n$ 
containing (at least) $8$ roots in $\Phi_{B_n}$,
then $P$ is standard and has exactly $8$ roots ($4$ long and $4$ short).
\item If $n\geq 3$ and $Q$ is a $3$-space in $\R^n$ containing (at least) $12$ long roots in $\Phi_{B_n}$,
then $Q$ is either standard or there exist a four-dimensional standard subspace $U$ with 
basis $\{e_i,e_j,e_l,e_s\}$ and $\delta_1,\delta_2,\delta_3\in\{\pm 1\}$ such that $Q$ is the orthogonal complement of 
$v:=e_i+\delta_1 e_j+\delta_2e_l+\delta_3e_s$ in $U$.
\end{enumerate}
\end{prop}

\begin{proof}
The assertions (i),(ii),(iii) follow easily from \eqref{rootsys}.

For (iv), we pick two roots $\alpha,\beta$ that span $P$. If $P$ is not standard,
then the pair has either one long root or two long roots. In the first case, we may 
assume $\alpha=e_i$, $\beta= e_j+\epsilon e_k$, 
$\epsilon\in\{\pm1\}$ by (ii), and that $\{e_i,e_j,e_k\}$ spans a $3$-space.
Since $P$ contains $8$ roots, there exists a root $\gamma\in P$ not belonging to $\{\pm \alpha,\pm\beta\}$.
Hence, $\gamma= a\alpha +b\beta$ for some non-zero $a,b\in\R$ but this contradicts (i).
The second case is impossible by (iii) and the assumption that $P$ contains $8$ roots. 
Therefore, we conclude that $P$ is standard and has exactly $8$ roots.

For (v), we pick three long roots $\alpha= e_i+\epsilon_1 e_j,
\beta= e_k+\epsilon_2 e_l,\gamma= e_r+\epsilon_3 e_s$ that span $Q$ by (ii),
where $\epsilon_1,\epsilon_2,\epsilon_3\in\{\pm1\}$.
Let $d$ be the dimension of the space spanned by $\{e_i,e_j,e_k,e_l,e_r,e_s\}$
and $P_{\alpha,\beta}$, $P_{\alpha,\gamma}$, $P_{\beta,\gamma}$ 
be the planes spanned by respectively $\{\alpha,\beta\}$, $\{\alpha,\gamma\}$, and $\{\beta,\gamma\}$.
If $d=5$ (resp. $6$), it follows from linear independence of $\{\alpha,\beta,\gamma\}$ and (i) that 
every root in $Q$ belongs to $P_{\alpha,\beta}\cup P_{\alpha,\gamma}\cup P_{\beta,\gamma}$,
and then (iii) implies that $Q$ has exactly $8$ (resp. $6$) roots. But this is absurd.
If $d=3$, then $Q$ is standard. Now suppose $d=4$. 
If $P_{\alpha,\beta}$ is standard, then $\{i,j\}=\{k,l\}$, $\{i,j\}\cap \{r,s\}=\emptyset$, and (i),(iii) imply that 
$Q$ has exactly $6$ long roots, which is absurd.
Hence, all planes $P_{\alpha,\beta}$, $P_{\alpha,\gamma}$, $P_{\beta,\gamma}$ are not standard.
Without loss of generality, we assume $\{i,j\}\cap \{k,l\}=\{k\}$ 
and $\{i,j\}\cap \{r,s\}=\{r\}$. Then $Q$ is the orthogonal complement of 
$U$ (spanned by $\{e_i,e_j,e_l,e_s\}$) of some $v=e_i+\delta_1 e_j+\delta_2e_l+\delta_3e_s$.
\end{proof}

\begin{remark}
The orthogonal complement of $v=e_i+\delta_1 e_j+\delta_2e_l+\delta_3e_s$ in $U$ in (v) above contains exactly $12$ (long) roots.
\end{remark}

\begin{prop}\label{BnWeyl} 
Let $W_{B_n}$ denote the Weyl group of $B_n$.
\begin{enumerate}[(i)]
\item Let $\mathbb{H}_n:=[-1,1]^n\subset\R^n$ be the standard $n$-dimensional hypercube. 
Then $W_{B_n}$ is 
equal to $\{\phi\in O(\R^n):~\phi~\text{preserves}~\mathbb{H}_n\}$ and satisfies
$$0\to (\Z/2\Z)^n\to W_{B_n}\to S_n\to 0,$$
a (split) short exact sequence of groups where $S_n$ is the symmetric group of $n$ elements.
\item Any reflection $\sigma\in W_{B_n}$ is induced from a root $\alpha\in B_n$, i.e., the $-1$-eigenspace of $\sigma$ is spanned by $\alpha$.
\end{enumerate}
\end{prop}

\begin{proof}
The assertion (i) is well-known (see e.g., \cite[$\mathsection$~12.1]{Hum78}), and
(ii) follows from \cite[Proposition~1.14]{Hum90}.
\end{proof}

\subsection{Rectangular and hypercubic subsets}\label{s2.2}

\begin{defi}\label{rectangle}
For $d\in\Z_{\geq 0}$, define $Z_d:=\{-d,-d+2,-d+4,...,d-2,d\}\subset\Z$.
Let $V_\R$ be an $n$-dimensional real vector space.
\begin{enumerate}[(1)]
\item A subset $\Xi\subset V_\R$ is said to be 
\emph{rectangular} if there is an $\R$-isomorphism $\iota:V_\R\to \R^n$ such that 
$\iota(\Xi)=Z_{d_1}\times Z_{d_2}\times\cdots\times Z_{d_n}$
for some $(d_1,d_2,...,d_n)\in\Z_{\geq0}^n$. In this case, we say that 
$\Xi$ is isomorphic to $Z_{d_1}\times Z_{d_2}\times\cdots\times Z_{d_n}$ and the multiset 
$$\mathscr L_\Xi:=\{\ell_i:=|Z_{d_i}|=d_i+1:~1\leq i\leq n\}$$ 
is called the \emph{set of lengths}\footnote{This notation is well-defined by Proposition \ref{matching} below.} of $\Xi$. 
\item If moreover $d_1=d_2=\cdots=d_n$, then $\Xi$ is said to be \emph{hypercubic}
and $\ell=d_1+1$ is called the \emph{length} of $\Xi$.
\end{enumerate}
\end{defi}

By definition, we have $|\Xi|=\prod_{i=1}^n \ell_i$.

\begin{prop}\label{matching}
Let $0\leq d_1\leq d_2\leq \cdots \leq d_n$ and $0\leq d_1'\leq d_2'\leq \cdots \leq d_n'$ be integers.
Suppose $\phi:\R^n\to\R^n$ is an isomorphism taking
$\Xi= Z_{d_1}\times\cdots\times Z_{d_n}$ to $\Xi'= Z_{d_1'}\times\cdots\times Z_{d_n'}$.
The following assertions hold.
\begin{enumerate}[(i)]
\item The numbers of zeros in the multisets $\{d_1,...,d_n\}$ and $\{d_1',...,d_n'\}$ are equal.
\item For all $1\leq i\leq n$, one has $d_i=d_i'$.
\item If $d_1,d_1'>0$, then $\phi(e_i)=\pm e_j$ where $d_i=d_j'$. 
\end{enumerate}
\end{prop}

\begin{proof}
Since the sum of the number of zeros in $\{d_1,...,d_n\}$ (resp. $\{d_1',...,d_n'\}$) 
and the dimension of $\mathrm{Span}\Xi$ (resp. $\mathrm{Span}\Xi'$) is equal to $n$,
the assertion (i) follows from $\mathrm{Span}\Xi\simeq \mathrm{Span}\Xi'$.

For (ii), we may assume $d_1,d_1'>0$ by (i). For a subset $S\subset\R^n$, define $m(S):=\{\frac{s_1+s_2}{2}:~ s_1,s_2\in S\}$ as the set
of mid-points. If $\phi:\R^n\to\R^n$ is an isomorphism, then $\phi(m(S))=m(\phi(S))$.
Define $M_d:=m(Z_{d})=\{x\in\Z: ~-d\leq x\leq d\}$ for $d\in\N$ and note that 
\begin{itemize}
\item $m(\Xi)=m(Z_{d_1})\times m(Z_{d_2})\times\cdots\times m(Z_{d_n})=M_{d_1}\times M_{d_2}\times\cdots\times M_{d_n}$, and
\item $\phi(M_{d_1}\times M_{d_2}\times\cdots\times M_{d_n})=M_{d_1'}\times M_{d_2'}\times\cdots\times M_{d_n'}$.
%\item $\pi_i(m(\Xi))\subset m(\Xi)$, where $\pi_i:\R^n\to\R$ denotes the projection to the $i$th coordinate.
\end{itemize}

Let $0<k_1<k_2<\cdots< k_r$ and $n_1,n_2,...,n_r$ be positive integers such that 
\begin{equation}\label{product1}
m(\Xi)=M_{d_1}\times M_{d_2}\times\cdots\times M_{d_n}= M_{k_1}^{n_1}\times M_{k_2}^{n_2}\times 
\cdots\times M_{k_r}^{n_r}.
\end{equation}

Let $0<\ell_1<\ell_2<\cdots< \ell_s$ and $m_1,m_2,...,m_s$ be positive integers such that 
\begin{equation}\label{product2}
m(\Xi')=M_{d_1'}\times M_{d_2'}\times\cdots\times M_{d_n'}= M_{\ell_1}^{m_1}\times M_{\ell_2}^{m_2}\times \cdots\times M_{\ell_s}^{m_s}.
\end{equation}

\begin{lemma}\label{lem_phi zero first part}
In equations \eqref{product1} and \eqref{product2}, $k_r=\ell_s$ and $n_r=m_s$.
Moreover, 
\begin{equation}\label{equalpart}
\phi(0_{\R^{n-n_r}}\times M_{k_r}^{n_r})=0_{\R^{n-m_s}}\times M_{\ell_s}^{m_s}.
\end{equation}
\end{lemma}

\begin{proof}
Consider a standard basis vector $e_i$ in $0_{\R^{n-n_r}}\times M_{k_r}^{n_r}$.
Then
\begin{equation}\label{nonzero}
\phi(e_i),\phi(k_r e_i)\hspace{.1in} \text{are non-zero in}\hspace{.1in} M_{\ell_1}^{m_1}\times M_{\ell_2}^{m_2}\times \cdots\times M_{\ell_s}^{m_s}.
\end{equation}
It follows that 
$k_r\leq \ell_s$ and thus $k_r=\ell_s$ by symmetry.
Since $\ell_1,...,\ell_{s-1}<\ell_s=k_r$, the assertion \eqref{nonzero}
implies that $\phi(e_i)\in 0_{\R^{n-m_s}}\times M_{\ell_s}^{m_s}$
and thus $\phi(0_{\R^{n-n_r}}\times M_{k_r}^{n_r})\subset 0_{\R^{n-m_s}}\times M_{\ell_s}^{m_s}$.
We obtain \eqref{equalpart} by symmetry and $n_r=m_s$.
\end{proof}

\begin{lemma}\label{lem_phi zero second part}
In equations \eqref{product1} and \eqref{product2}, one has
\begin{equation}\label{equalpart2}
\phi(M_{k_1}^{n_1}\times \cdots\times M_{k_{r-1}}^{n_{r-1}}\times 0_{\R^{n_r}})=M_{\ell_1}^{m_1}\times 
\cdots\times M_{\ell_{s-1}}^{m_{s-1}}\times 0_{\R^{m_s}}.
\end{equation}
\end{lemma}

\begin{proof}
Suppose $v\in M_{k_1}^{n_1}\times \cdots\times M_{k_{r-1}}^{n_{r-1}}\times 0_{\R^{n_r}}$. If $\phi(v)$ in \eqref{product2} is not contained in the right-hand side of \eqref{equalpart2},
then \eqref{equalpart} implies that
there exists $v'\in 0_{\R^{n-n_r}}\times M_{k_r}^{n_r}$ such that $\phi(v+v')$ lies outside \eqref{product2},
which is absurd. We obtain \eqref{equalpart2} by symmetry.
\end{proof}

Since Lemma~\ref{lem_phi zero first part} and \ref{lem_phi zero second part} imply that the two direct products in \eqref{product1} and \eqref{product2} are identical
and $\phi$ preserves each component,
we obtain the assertion (ii) and it suffices to consider $\Xi=\Xi'=Z_d^n$ for (iii).

We would like to check that $\phi$ preserves $\{\pm e_i:~1\leq i\leq n\}\subset M_d^n$.
The case $n=1$ is trivial, so assume $n \geq 2$. 
Since $de_i\in M_d^n$, we obtain $d\phi(e_i)=\phi(de_i)\in M_d^n\subset\Z^n$.
Thus, if  $\phi(de_i)=(x_1,x_2,...,x_n)\in\Z^n$ then
\begin{itemize}
\item $|x_\ell|\in\{0,d\}$ for all $1\leq \ell\leq n$ and
\item not all $x_\ell$ are zero.
\end{itemize}
Suppose there are two coordinates $x_\ell$ with absolute value $d$. By the pigeon-hole principle,
there exists $e_j\neq e_i$ such that the absolute value 
of some coordinate of $de_i\pm de_j$ is $2d$. But this contradicts that $\phi(de_i\pm de_j)\subset M_d^n$.
\end{proof}

\begin{defi}
Let $0< d_1<d_2<\cdots <d_r$ and $n_1,n_2,...,n_r$ be positive integers, where $r\in\N$.
Let $n:=n_1+n_2+\cdots+ n_r$ and define the automorphism groups
$$\Omega_i:=\{\phi\in\GL_{n_i}(\R):~\phi~\text{preserves}~Z_{d_i}^{n_i}\subset\R^{n_i}\}$$ for $1\leq i\leq r$
and 
$$\Omega:=\{\phi\in\GL_n(\R):~\phi~\text{preserves}~ Z_{d_1}^{n_1}\times Z_{d_2}^{n_2} \times\cdots\times
Z_{d_r}^{n_r}\subset\R^n\}.$$
\end{defi}

\begin{cor}\label{auto}
Under the standard Euclidean metric on $\R^n$, 
there are natural identifications 
\begin{enumerate}[(i)]
\item $\Omega_i=W_{B_{n_i}}$ for all $i$ and
\item $\Omega=\Omega_1\times\Omega_2\times\cdots\times\Omega_r$.
\end{enumerate}
In particular, $\Omega$ is orthogonal.
\end{cor}

\begin{proof}
Since $\Omega_i$ has $2^{n_i} n_i!$ elements and preserves $\mathbb{H}_{n_i}$ the hypercube by 
Proposition \ref{matching}(iii), the assertion (i) holds by Proposition \ref{BnWeyl}(i).
The assertion (ii) follows from Proposition \ref{matching}(iii).
\end{proof}

\subsection{Rectangular and hypercubic representations}\label{s2.3}

Let $\Lieg$ be a complex semisimple Lie algebra of rank $n$. Fix a Cartan subalgebra $\Liet$ of $\Lieg$ and 
denote by $\Lambda_\Lieg$ the weight lattice of $\Lieg$ with respect to $\Liet$, regarded as a rank $n$ lattice in $\Lambda_\Lieg\otimes\R\simeq\R^n$.

\begin{defi}\label{rectrepn}
A Lie algebra homomorphism $\psi:\Lieg\to\End(V)$ of the semisimple Lie algebra $\Lieg$ is determined
by its \emph{formal character} $\psi|_\Liet$, which corresponds to a multiset $\Xi$ of weights in $\Lambda_\Lieg$.
We identify the formal character with the multiset $\Xi$.
\begin{enumerate}[(1)]
\item We say that $\psi$ is \emph{rectangular} (resp. \emph{hypercubic}) if 
 each weight of $\Xi$ is of multiplicity one and
the subset $\Xi\subset\Lambda_\Lieg\otimes\R$ is rectangular (resp. hypercubic). 
\item In this case, 
we say that the weights in $\Xi$ are rectangular (resp. hypercubic) and we 
define the \emph{set of lengths} $\mathscr L_\psi$ of $\psi$ to be $\mathscr L_\Xi$ 
(resp. the \emph{length} of $\psi$ to be the length of $\Xi$).
\item A rectangular representation $(\Lieg,\psi)$ is said to be \emph{decomposable}
if there exist rectangular representations $(\Lieg_1,\psi_1)$ and $(\Lieg_2,\psi_2)$
such that $\Lieg=\Lieg_1\times\Lieg_2$ and $\psi\simeq\psi_1\otimes\psi_2$ as an external tensor product.
If such a decomposition does not exist, we say that $(\Lieg,\psi)$ is \emph{indecomposable}.
\end{enumerate}
\end{defi}

\begin{remark}
If $\psi$ above is faithful and rectangular, then the lengths of $\Xi$ are at least $2$ because 
$\Xi$ spans $\Lambda_\Lieg\otimes\R$.
\end{remark}

\begin{example}\label{example}
Here are some indecomposable hypercubic representations.
\begin{enumerate}[(1)]
\item The one-dimensional trivial representation is hypercubic of length $1$.
\item Any irreducible representation of $A_1$.
\item The direct sum $V_1\oplus V_2$ of two irreducible representations of $A_1$ whose dimensions differ by $1$.
\item (\cite[$\mathsection20.1$]{FH91}) The spin representation $\mathrm{Spin}$ of $B_n$ for $n\in\N$ (resp. $D_n$ for $n\in\N_{\geq 2}$)
is of dimension $2^n$. It 
is hypercubic of length $2$ with formal character $\Xi$ equal to
$$\{(\pm 1/2,\pm 1/2,...,\pm 1/2)\in\R^n\}$$
if the root system is identified with \eqref{rootsys} in $\mathsection\ref{s2.1}$. Note that $(B_n,\mathrm{Spin})$ is irreducible
and $(D_n,\mathrm{Spin})=(D_n,\mathrm{Spin}^+\oplus \mathrm{Spin}^-)$ is the sum 
of two irreducible half-spin representations.
\item (\cite[$\mathsection20.3$]{FH91}) The outer automorphism group $\mathrm{Out}(D_4)$ of $D_4$ is isomorphic to $S_3$. Moreover, $\mathrm{Out}(D_4)$ acts faithfully on 
$\{\mathrm{Std}, \mathrm{Spin}^+, \mathrm{Spin}^-\}$ by composition. Since $(D_4,\mathrm{Spin})$ is hypercubic of length $2$, we obtain that $(D_4, \mathrm{Std}\oplus \mathrm{Spin}^+)$ and $(D_4, \mathrm{Std}\oplus \mathrm{Spin}^-)$ are also hypercubic of length $2$.
\item The direct sum $\mathrm{Std}\oplus\mathrm{Spin}$ representation of $B_2$ is hypercubic of length $3$
with formal character $\Xi$ equal to 
$$\{(\pm 1,0),(0,\pm 1), (0,0)\}\cup \{(\pm 1/2,\pm 1/2)\}\subset\R^2$$
if the root system is identified with \eqref{rootsys} in $\mathsection\ref{s2.1}$. Note that the restriction of such a representation
to the subalgebra $A_1\times A_1\subset B_2$ 
is the external tensor product $(A_1, \mathrm{Std}\oplus \mathbb 1)\otimes (A_1, \mathrm{Std}\oplus \mathbb 1)$.
\end{enumerate}
\end{example}

\begin{remark}
Theorem \ref{main1} asserts that any faithful rectangular representation of a semisimple Lie algebra
is the external tensor product of these indecomposable hypercubic representations in a unique way.
\end{remark}

\subsection{Reduction to the hypercubic case}\label{s2.4}
Let $I$ be a finite index set. For each $i \in I$, let $\Lieg_i$ be a complex semisimple Lie algebra with  corresponding weight lattice $\Lambda_{\Lieg_i}$ (with respect to a Cartan subalgebra $\Liet_i$) and Weyl group $W_{\Lieg_i}$. 
The weight lattice $\Lambda_\Lieg$ of the product 
$\Lieg:=\prod_{i\in I}\Lieg_i$ (with respect to Cartan subalgebra $\prod_{i\in I}\Liet_i$) is identified as
$\bigoplus_{i\in I}\Lambda_{\Lieg_i}$. The following lemmas are well-known.

\begin{lemma}\label{wk1}
The action of the Weyl group $W_{\Lieg}$ of $\Lieg$ on $\Lambda_{\Lieg}\otimes\R$ 
is identified as the action of $\prod_{i\in I}W_{\Lieg_i}$ on 
$\bigoplus_{i\in I}(\Lambda_{\Lieg_i}\otimes\R)$.
If $\Lieg_i$ is simple, then the action of $W_{\Lieg_i}$ on $\Lambda_{\Lieg_i}\otimes\R$
is irreducible and non-trivial.
\end{lemma}

\begin{lemma}\label{wk2}
Suppose $\Lambda_{\Lieg}\otimes\R$ 
is equipped with a positive definite metric $||\cdot||$ such that $W_{\Lieg}$ is orthogonal.
The following assertions hold.
\begin{enumerate}[(i)]
\item The direct sum decomposition 
$\Lambda_{\Lieg}\otimes\R=\bigoplus_{i\in I}(\Lambda_{\Lieg_i}\otimes\R)$
is orthogonal.
\item If $\Lieg_i$ is a simple, the restriction of $||\cdot||$ to the subspace $\Lambda_{\Lieg_i}\otimes\R$
is, up to a positive scalar, induced from the Killing form of $\Lieg_i$.
\end{enumerate}
\end{lemma}

The following general result is crucial to decomposing representations.

\begin{prop}\label{decompose}
Let $\Lieg_1,\Lieg_2$ be semisimple Lie algebras and 
$\psi: \Lieg_1\times\Lieg_2\to \End(V)$ be a representation such that the formal character 
$\Xi\subset \Lambda_{\Lieg_1\times\Lieg_2}=\Lambda_{\Lieg_1}\oplus\Lambda_{\Lieg_2}$ 
is a direct product of multisets:
\begin{equation}\label{dp}
\Xi=\Xi_1\times\Xi_2:=\{(w_1,w_2)\in\Lambda_{\Lieg_1}\oplus\Lambda_{\Lieg_2}:~ w_1\in\Xi_1,~w_2\in\Xi_2\},
\end{equation}
where $\Xi_1$ (resp. $\Xi_2$) is a finite multiset in $\Lambda_{\Lieg_1}$ (resp. $\Lambda_{\Lieg_2}$).
Then there exist representations $\psi_i:\Lieg_i\to\End(V_i)$ with formal character $\Xi_i$ for $i=1,2$
such that $\psi$ is isomorphic to 
\begin{equation*}\label{tensor}
\psi_1\otimes\psi_2:\Lieg_1\times\Lieg_2\to\End(V_1\otimes V_2).
\end{equation*}
\end{prop}

\begin{proof}
By \eqref{dp}, the multiplicity of each weight in the formal character of the restriction $\psi|_{\Lieg_1}$ is divisible by $|\Xi_2|\in\N$.
By the highest weight theory of semisimple Lie algebra, there is a representation $\psi_1$ of $\Lieg_1$ such that  
$\psi|_{\Lieg_1}=\psi_1^{\oplus |\Xi_2|}$. It follows from \eqref{dp} 
that the formal character of $\psi_1$ is $\Xi_1$.
Similarly, we find a representation $\psi_2$
of $\Lieg_2$ with formal character $\Xi_2$.
These imply that the external tensor product $\psi_1\otimes\psi_2$
is isomorphic to $\psi$.
\end{proof}

The following result enables us to reduce the problem to the hypercubic
case.

\begin{prop}\label{reduction1}
Let $\Lieg$ be a complex semisimple Lie algebra
and $(\Lieg,\psi)$ be a faithful rectangular representation with formal character 
$\Xi\subset\Lambda_\Lieg\otimes\R$ isomorphic to 
\begin{equation}\label{normal}
Z_{d_1}^{n_1}\times Z_{d_2}^{n_2} \times\cdots\times Z_{d_r}^{n_r}\subset \R^n,
\end{equation}
where $0<d_1<d_2<\cdots< d_r$ and $n_1,n_2,...,n_r$ are positive integers, and 
$n=n_1+n_2+\cdots+n_r$.
Then $\Lieg=\prod_{1\leq i\leq r}\Lieg_i$ is a product of semisimple Lie algebras $\Lieg_i$
and $\psi=\bigotimes_{1\leq i\leq r}\psi_i$ is the external tensor product of the faithful hypercubic 
representations $(\Lieg_i,\psi_i)$ with formal characters  $\Xi_i\subset \Lambda_{\Lieg_i}\otimes\R$ isomorphic
to $Z_{d_i}^{n_i}\subset\R^{n_i}$.
\end{prop}

\begin{proof}
Write $\Lieg=\prod_{j\in J}\Lieq_j$ as a product of simple Lie algebras.
Without loss of generality, suppose $\Xi\subset\Lambda_\Lieg\otimes\R$ is equal to \eqref{normal}
and equip $\R^n$ as the standard Euclidean metric. Since $W_{\Lieg}$ preserves the formal 
character $\Xi$, the Weyl group $W_{\Lieg}\subset \Omega$ is orthogonal by Corollary \ref{auto}.
For each $1\leq i\leq r$, define the $n_i$-dimensional subspace
$$Y_i:=\text{Span}_\R\{e_j:~ n_1+\cdots+n_{i-1}< j\leq n_1+\cdots+n_{i-1}+n_i\}.$$
It follows that $Y_i$ is a $W_\Lieg$-subrepresentation of $\mathbb R^n$ (resp. $W_{\Lieq_j}$ for $j\in J$) 
by Corollary \ref{auto}(ii). As $\Lieq_j$ is simple, $\Lambda_{\Lieq_j}\otimes\R$
is a subspace of a unique $Y_i$ by Lemma \ref{wk1}. Therefore, $J$ is
partitioned into $J_1\cup J_2\cup \cdots \cup J_r$ such that 
$Y_i=\bigoplus_{j\in J_i} (\Lambda_{\Lieq_j}\otimes\R)$ (by dimension considerations).
Let $\Lieg_i=\prod_{j\in J_i}\Lieq_j$ for $1\leq i\leq r$. Since $\R^n=Y_1\oplus Y_2\oplus\cdots\oplus Y_r$
and $Y_i$ is a standard subspace for all $i$, 
we obtain $\Xi_i:=Z_{d_i}^{n_i}\subset\Lambda_{\Lieg_i}\subset \Lambda_{\Lieg_i}\otimes\R=Y_i$ for all $i$.
Therefore, we are done by applying Proposition \ref{decompose}.
\end{proof}

\subsection{Extracting $A_1$-factors and excluding $A_2$, $A_r$ ($r\geq 4$), and exceptional factors}\label{s2.5}
Let $\psi:\Lieg\to\End(V)$ be a faithful hypercubic representation 
of a semisimple Lie algebra $\Lieg$ and suppose the formal character $\Xi\subset\Lambda_{\Lieg}\otimes\R$ 
is equal to $Z_d^n\subset\R^n$ (Euclidean space).

\begin{prop}\label{extract}
If $\Lieg=\Lieg_1\times\Lieg_2$ where $\Lieg_1$ is a product of $A_1$-factors and $\Lieg_2$ has no $A_1$-factors,
then the following assertions hold.
\begin{enumerate}[(i)]
\item $\Lambda_{\Lieg_1}\otimes\R$ and $\Lambda_{\Lieg_2}\otimes\R$ are standard in $\R^n$.
\item There exist faithful hypercubic representations $\psi_1:\Lieg_1\to\End(V_1)$ and $\psi_2:\Lieg_2\to \End(V_2)$
such that $\psi=\psi_1\otimes\psi_2$.
\item Moreover, $\Lieg_1=\prod_{1\leq j\leq s}\Lieg_{1,j}$ and 
$\psi_1=\bigotimes_{1\leq j\leq s}\psi_{1,j}$ is the external tensor product of the faithful hypercubic 
representations of $\Lieg_{1,j}$ for $1\leq j\leq s$, 
where each $\Lieg_{1,j}$ is either $A_1$ or $A_1\times A_1$.
\end{enumerate}
\end{prop}

\begin{proof}
We prove (i),(ii),(iii) simultaneously by induction on $n$. It is obviously true when $n=1$.
Suppose $\Lieg=A_1\times\Lieg'$ and let $\sigma$ be the reflection in $W_{A_1}\subset W_{\Lieg}$.
Since $W_{\Lieg}$ preserves $Z_d^n$ and thus the hypercube $\mathbb{H}_n$,
it is a subgroup of $W_{B_n}$ (the automorphism group of $\mathbb{H}_n$) by Proposition \ref{BnWeyl}(i).
Then Proposition \ref{BnWeyl}(ii) asserts that $\sigma$ is induced from a root $\alpha$ of $B_n$.
We consider two cases.

If $\alpha$ is a short root, then $\Lambda_{A_1}\otimes \R$ and thus $\Lambda_{\Lieg'}\otimes \R$ (by Lemma \ref{wk2}(i)) are standard subspaces of $\R^n$. By Proposition \ref{decompose},
$\psi$ is the external tensor product of a representation of $A_1$ and a (hypercubic) representation 
$\psi'$ of $\Lieg'$.
The result follows from the induction hypothesis on $\psi'$.

If $\alpha$ is a long root, then $W_{\Lieg'}$ preserves the fixed part
$\mathbb{H}_n^\sigma$ 
of the hypercube under $\sigma$ by Lemmas \ref{wk1} and \ref{wk2}.
Suppose the long root $\alpha=e_1+e_2$ without loss of generality. Then 
\begin{equation}\label{fixpt}
\mathbb{H}_n^\sigma=\{(x_1,x_2)\in\R^2:~ x_1+x_2=0,~|x_1|,|x_2|\leq 1\}\times \mathbb{H}_{n-2}=:L\times \mathbb{H}_{n-2}.
\end{equation}
Since the line segment $L$ has length $2\sqrt{2}$ different from $2$ (length of a side of $\mathbb{H}_{n-2}$),
$W_{\Lieg'}$ preserves each factor in \eqref{fixpt}. Thus, $\text{Span}_\R L$ is a
one-dimensional subrepresentation of $W_{\Lieg'}$.  
By Lemma \ref{wk2}(i) and dimension considerations, we obtain $\Lieg'=A_1\times\Lieg''$
and $W_{A_1}$ acts non-trivially on the line.
It follows that $\Lambda_{A_1\times A_1}\otimes\R=\text{Span}_\R\{e_1,e_2\}$
and $\Lambda_{\Lieg''}\otimes\R= \text{Span}_\R\{e_i: ~i>2\}$
are standard subspaces of $\R^n$.
By Proposition \ref{decompose},
$\psi$ is the external tensor product of a hypercubic representation of $A_1\times A_1$ and a hypercubic representation 
$\psi''$ of $\Lieg''$.
We are done by the induction hypothesis on $\psi''$.
\end{proof}

Given a simple Lie algebra $\Lieg$, we list some equal-rank subalgebra $\Lieg'\leq \Lieg$ (using \cite[Table 5]{GOV94}):
\begin{enumerate}
\item $A_1^{r}\leq B_{r}$ for $r\geq 2$.
\item $A_1^r\leq C_r$ for $r\geq 3$.
\item $A_1^{2(r-1)}\times A_3\leq D_{2r+1}$ for $r\geq 1$.
\item $A_1^{2r}\leq D_{2r}$ for $r\geq 2$.
\item $A_2\times A_2\times A_2\leq E_6$, 
\item $A_7\leq E_7$, 
\item $A_8\leq E_8$,
\item $A_2\times A_2\leq F_4$,
\item $A_2\leq G_2$.
\end{enumerate}

\begin{prop}\label{exclude}
Let $\Lieq$ be a simple factor of $\Lieg$. Then $\Lieq$ cannot be of 
type $A_2$, $A_r$ for any $r\geq 4$, and exceptional type.
\end{prop}

\begin{proof}
If $\Lieg'$ is a semisimple subalgebra of $\Lieg$ of equal-rank, then $\psi|_{\Lieg'}$ is also rectangular.
Hence, we may assume each simple factor of $\Lieg$ is of type $A$ by the list above.
By Proposition \ref{extract}, we may further assume that $\Lieg$ has no $A_1$-factors.
To prove the result, it suffices show that $\Lieg$ has only $A_3$-factors.

Suppose $\Lieg=A_3^r\times \Lieq_1\times\cdots\Lieq_k$
for some $k\in\N$, where each $\Lieq_i$ is equal to $A_{n_i}$ 
such that $n_i\in \{2\}\cup\N_{\geq 4}$.
It follows from Lemma \ref{wk1} and Corollary \ref{auto} that 
$$S_3^r\times \prod_{i=1}^k S_{n_i+1}\leq S_4^r\times \prod_{i=1}^k S_{n_i+1}
=W_{A_3}^r\times\prod_{i=1}^k W_{\Lieq_i}=W_\Lieg\subset \Omega=W_{B_n}.$$
Since the symmetric group $S_m$ does not have non-trivial normal $2$-subgroup 
for $m\in\{3\}\cup\N_{\geq 5}$, the homomorphism (via the inclusions above and Proposition \ref{BnWeyl}(i))
$$S_3^r\times \prod_{i=1}^k S_{n_i+1}\hookrightarrow W_{B_n}\to S_n$$
is injective. Since the rank of $\Lieg$ is $n=3r+n_1+\cdots+n_k$, 
we obtain a contradiction by Proposition \ref{symgp} below, which 
asserts that $3r+(n_1+1)+\cdots+(n_k+1)\leq n$.
Therefore, we conclude that $\Lieg$ has only $A_3$-factors.
\end{proof}

\begin{prop}\label{symgp}
Let $i: S_{m_1}\times S_{m_2}\times\cdots\times S_{m_t}\to S_n$ be an injective homomorphism such that 
$m_i>1$ for all $1\leq i\leq t$. Then $m_1+m_2+\cdots+m_t\leq n$.
\end{prop}

\begin{proof}
For a finite group $G$, let $\mu(G)$ denote the smallest positive integer $m$ such that there exists an injective homomorphism $G \hookrightarrow S_m$. Wright \cite{Wri75} proved that
$$\mu(S_{m_1}\times S_{m_2}\times\cdots\times S_{m_t})=\sum_{i=1}^t \mu(S_{m_i})=\sum_{i=1}^t m_i,$$
which completes the proof.
\end{proof}

\subsection{Proof of Theorem \ref{main1}(i)}\label{s2.6}
By Propositions \ref{reduction1}, \ref{extract}(ii) and \ref{exclude},
 we may assume that $\psi$ is hypercubic 
and $$\Lieg=\Lieq_1\times\Lieq_2\times\cdots\times\Lieq_k$$ is a product 
of simple Lie algebras such that $\Lieq_i$ is of type 
$B_m$ ($m\geq 2$) or $C_m$ ($m\geq 3$) or $D_m$ ($m\geq 3$ and $D_3=A_3$).
Suppose $\Xi\subset\Lambda_\Lieg\otimes\R$ is equal to $Z_d^n\subset\R^n$.
By Proposition \ref{decompose}, it suffices to show that 
each subspace $\Lambda_{\Lieq_i}\otimes\R\subset\R^n$
is standard.

Let $\Lieq$ be a simple factor of $\Lieg$. Then $W_{\Lieq}$ is a subgroup of $W_{\Lieg}\subset \Omega=W_{B_n}$
by Lemma \ref{wk1} and Corollary \ref{auto}. For a reflection $\sigma\in W_{\Lieq}$, denote the line 
of $-1$-eigenspace of $\sigma$ by $L_{\sigma}\subset\R^n$. We record the following observations.

\begin{lemma}\label{observe}
The following assertions hold for the subspace $\Lambda_{\Lieq}\otimes\R\subset\R^n$.
\begin{enumerate}[(i)]
\item The subspace $\Lambda_{\Lieq}\otimes\R$ 
is spanned by $\{L_\sigma:~\sigma~\text{is a reflection in }W_{\Lieq}\}$.
\item The subspace $\Lambda_{\Lieq}\otimes\R$ 
is spanned by some roots in $\Phi_{B_n}\subset\R^n$ in \eqref{rootsys}.
\end{enumerate}
\end{lemma}

\begin{proof}
The first assertion is obvious, and the second follows from the first assertion and Proposition \ref{BnWeyl}(ii).
\end{proof}

If $\Lieq$ is of type $B_m$ ($m\geq 2$) or $C_m$ ($m\geq 3$), 
then Lemmas \ref{wk2}(ii), \ref{observe}(i), and \eqref{rootsys} 
imply that $\Lambda_{\Lieq}\otimes\R$ is spanned by 
certain planes $P$ such that each $P$ contains four distinct lines $L_\sigma$ 
(coming from the configuration of a $B_2$ root system). 
Since the four lines correspond to eight roots of $B_n$ by Proposition \ref{BnWeyl}(ii), each such $P$ is standard by Proposition \ref{Bnroot}(iv) and thus
$\Lambda_{\Lieq}\otimes\R$ is standard by Lemma \ref{std}.

Now, we suppose $\Lieq$ is of type $D_m$ ($m\geq 3$) and it suffices to show that $\Lambda_{\Lieq}\otimes\R$ is standard. Note the following equal-rank subalgebras $\Lieq'\subset\Lieq$ where $m$ is either $2r+1$ or $2r$.
\begin{enumerate}[(1)]
\item $A_1^{2(r-1)}\times A_3\leq D_{2r+1}$ for $r\geq 1$.
\item $A_1^{2r}\leq D_{2r}$ for $r\geq 2$.
\end{enumerate}
Let $\Lieg'$ be the equal-rank subalgebra of $\Lieg$ obtained by 
restricting the simple factor %$\Lieg$ to $\Lieg'$ simple factor 
$\Lieq$ of $\Lieg$ to $\Lieq'$
and letting all other factors remain unchanged.
Since the restriction $\psi|_{\Lieg'}$ is hypercubic and all the $A_1$-factors
come from $\Lieq'$, Proposition \ref{extract}(i) implies
that $\Lambda_{A_1^{2(r-1)}}\otimes\R$ is standard if $m=2r+1$ (resp. $\Lambda_{A_1^{2r}}\otimes\R$ is standard if $m=2r$). Since $\Lambda_{\Lieq}\otimes\R=\Lambda_{\Lieq'}\otimes\R$ in $\R^n$,
it is standard if we can show that $\Lambda_{A_3}\otimes\R$ is standard.

It remains to consider the case $\Lieq_1=D_3=A_3$.
Lemmas \ref{wk2}(ii), \ref{observe}(i), and \eqref{rootsys} 
imply that $Q:=\Lambda_{\Lieq_1}\otimes\R$ is a $3$-space spanned by six distinct lines $L_\sigma$ 
such that the angle between any two distinct lines is either $90^\circ$ or $60^\circ/120^\circ$.
The six lines contain $12$ roots $\alpha$ of $B_n$ by Proposition \ref{BnWeyl}(ii).
If $\alpha\in L_\sigma$ is a short root, then the angle between $L_\sigma$
and any other of the five lines $L_{\sigma'}$ is a multiple of $45^\circ$ by \eqref{rootsys}.
It follows that $L_\sigma$ is perpendicular to $L_{\sigma'}$ as $60^\circ/120^\circ$
is not a multiple of $45^\circ$. Since $Q$ is a $3$-space, ten roots of $B_n$ lie on the plane
spanned by the five lines,
which is absurd (Proposition \ref{Bnroot}(iv)). 
Hence, the twelve roots $\alpha$ are all long roots of $B_n$.
If $Q$ is not standard, it follows from
Proposition \ref{Bnroot}(v) that there exist a standard $4$-space $U$ spanned by $\{e_i,e_j,e_l,e_s\}$
and $v=e_i+\delta_1e_j+\delta_2e_l+\delta_3e_s$ ($\delta_1,\delta_2,\delta_3\in\{\pm1\}$) such that $Q\subset U$ is the orthogonal complement of $v$ in $U$.
Then Lemma \ref{wk2}(i) implies that 
$$\mathrm{Span}_\R\{e_t:~1\leq t\leq n,~t\notin\{i,j,l,s\}\}\oplus \R v
=(\Lambda_{\Lieq_2}\otimes\R)\oplus\cdots\oplus(\Lambda_{\Lieq_k}\otimes\R).$$
But this equation is impossible by Lemma \ref{observe}(ii), Proposition \ref{Bnroot}(i),
and the fact that $v$ has four non-zero coordinates.
We conclude that $Q$ is standard.
\qed

\subsection{Proof of Theorem \ref{main1}(ii)}\label{s2.7}
By Propositions \ref{reduction1} and \ref{extract}(iii), the rectangular representation $\psi_1$ 
of $\Lieg_1$ is the external tensor product $\bigotimes_{j=1}^{s} \psi_{1,j}$
of the hypercubic representations $\psi_{1,j}$ of $\Lieg_{1,j}$,
where $\Lieg_{1,j}$ is either $A_1$ or $A_1\times A_1$.
If $\Lieg_{1,j}=A_1$, then Theorem \ref{main1}(ii)(a),(b) are the only possibilities
since formal character $\Xi_{1,j}$ of $\psi_{1,j}$ is rectangular and each weight is of multiplicity one.

Next, suppose $\Lieg_{1,j}=A_1\times A_1$ and $\psi_{1,j}$ is not the external tensor product of
two rectangular representations (otherwise, we reduce to the previous case).
Let $\Z^2\subset\R^2=\Lambda_{A_1\times A_1}\otimes\R$ 
be the weight lattice of $A_1\times A_1$ so that $\Z\times\{0\}$
(resp. $\{0\}\times \Z$) corresponds to the weight lattice of the first (resp. second) $A_1$.
If $d+1$ (at least $2$) is the length of the faithful hypercubic $\psi_{1,j}$, then 
the formal character is equal to %(Figure 2)
\begin{equation}\label{sq}
\{(x,y)\in\Z^2:~|x|+|y|\leq d~~\text{and}~~ x+y\equiv d~(\text{mod}~2)\}.
\end{equation}

\begin{figure}[h]
\ctikzfig{fig3}
\caption{\eqref{sq} for $d=4$}
\end{figure}

Suppose $\Z_{\geq 0}^2$ is the set of dominant weights.
It follows that $(0,d),(1,d-1),(2,d-2),....,(d,0)$ are highest weights of \eqref{sq}
and thus 
$$\bigoplus_{0\leq i\leq d} (A_1, \mathrm{Sym}^i(\mathrm{Std}))\otimes (A_1, \mathrm{Sym}^{d-i}(\mathrm{Std}))$$
is a subrepresentation of $\psi_{1,j}$. We then obtain
$$(d+1)^2= \sum_{0\leq i\leq d} (d+1)\leq \sum_{0\leq i\leq d} (i+1)(d-i+1) \leq |\Xi_{1,j}|=(d+1)^2,$$
which implies $d+1=(i+1)(d-i+1)$ for all $i$. By putting $i=1$, we obtain $d=1$ and 
$(A_1 \times A_1, \psi_{1,j})=(A_1 \times A_1, (\mathrm{Std} \otimes \mathbb 1) \oplus (\mathbb 1 \otimes \mathrm{Std}))$.\qed

\subsection{Proof of Theorem \ref{main1}(iii)}\label{s2.8}
\subsubsection{}
By Theorem \ref{main1}(i) and Proposition \ref{exclude}, 
it suffices to classify hypercubic representations $\psi$ of simple Lie algebra $\Lieq$,
where $\Lieq$ is $A_3$, or $B_n$ ($n\geq 2$), or $C_n$ ($n\geq 3$), or $D_n$ ($n\geq 4$).
Note that a hypercubic representation $\psi$ is \emph{weight multiplicity-free}, 
i.e., every weight has multiplicity one. 
We present the classification of irreducible weight multiplicity-free representations of simple Lie algebras, as given by Howe.

\begin{theorem}\cite[Theorem~4.6.3]{How92}\label{thm_WMF}
    The non-trivial irreducible weight multiplicity-free representations 
		of a simple Lie algebra are those on the following lists.
    \begin{enumerate}
        \item[$A_n$:] \begin{enumerate}[(i)]
				\item the alternating powers $\bigwedge^m(\C^{n+1})$ of the standard representation, and
				\item the symmetric powers $\mathrm{Sym}^m(\C^{n+1})$ and $\mathrm{Sym}^m(\C^{n+1})^\vee$
				of the standard representation and its dual.
				\end{enumerate}
        \item[$B_n$:] \begin{enumerate}[(i)]
				\item the standard representation on $\C^{2n+1}$, and
				\item the spin representation $\mathrm{Spin}$.
				\end{enumerate}
				\item[$C_n$:] \begin{enumerate}[(i)]
				\item the standard representation on $\C^{2n}$, and
				\item if $n=2$ or $3$, the last fundamental representation, on $\bigwedge_{\mathrm{prim}}^2(\C^4)$
				and $\bigwedge_{\mathrm{prim}}^3(\C^6)$ respectively (of dimensions $5$ and $14$).
				\end{enumerate}
				\item[$D_n$:] \begin{enumerate}[(i)]
				\item the standard representation on $\C^{2n}$, and
				\item the two half-spin representations $\mathrm{Spin}^+$ and $\mathrm{Spin}^-$.
				\end{enumerate}
       \item[$E$:] \begin{enumerate}[(i)]
				\item the two $27$-dimensional representations of $E_6$,  and
				\item the $56$-dimensional representation of $E_7$.
				\end{enumerate}
				\item[$G$:] \begin{enumerate}[(i)]
				\item the $7$-dimensional representation of $G_2$.
				\end{enumerate}
    \end{enumerate}
		There are no non-trivial weight multiplicity-free representations for $E_8$ and $F_4$.
\end{theorem}
 
\subsubsection{}
We begin by classifying the hypercubic representations for $A_3$. 
The weight lattice of $A_3$ is generated by the weights of the standard representation $\C^4$,
which are four vectors $f_1,f_2,f_3,f_4\in \R^3$ (Euclidean space) 
satisfying $f_1+f_2+f_3+f_4=0$. They can be seen as four vertices of the cube $[-1/2,1/2]^3\subset\R^3$. (Figure 5).

\begin{figure}[h]
\ctikzfig{fig4.1}
\caption{$f_1,f_2,f_3,f_4$ in $[-1/2,1/2]^3$}
\end{figure}

\begin{lemma}\label{lem_lattice line lemma}
    Suppose $\Xi=Z_{d_1}\times Z_{d_2}\times\cdots\times Z_{d_n}\subset \R^n$ is rectangular and $L_1, \ldots,L_r \subset \R^n$ 
		are $r$ affine lines (not necessarily passing through the origin) that satisfy the conditions below.
    \begin{enumerate}[(a)]
        \item $\min_{1\leq i\leq r} |L_i \cap \Xi|=m$,
        \item For any $1\leq j\leq n$, there exists $k \in\{1,2,\ldots,r\}$ such that $\pi_j(L_k)=\mathbb R$, where $\pi_j:\R^n\to\R$ denotes projection to the $j$th component.
    \end{enumerate} Then, $|\Xi| \geq m^n$.
\end{lemma}

\begin{proof}
   One has from $\pi_{1}(L_{i_{1}}), \ldots, \pi_{n}(L_{i_n}) =\mathbb R$ for some indices $1 \leq i_{1}, \ldots, i_n \leq r$ that
    \begin{align*}
        |Z_{d_j}|=|\pi_j(L_{i_j}) \cap \pi_j(\Xi)| \geq |\pi_j(L_{i_j} \cap \Xi)|=|L_{i_j} \cap \Xi| \geq m, \quad 1 \leq j \leq n,
    \end{align*}
    and consequently $|\Xi|\geq \prod_{1\leq j\leq n} |Z_{d_j}|\geq m^n$.
\end{proof}

\begin{prop}\label{lem_rectangular A_3}
    The only faithful hypercubic representation $V$ of $A_3$ is $\mathrm{Std} \oplus \mathrm{Std}^{\vee}$.
\end{prop}

\begin{proof}
Let $\Xi\subset\R^3$ be the formal character of $V$ that is hypercubic of length $\ell\geq 2$.
    By Theorem~\ref{thm_WMF}, if $V$ is hypercubic, then each irreducible factor $W$ of $V$ is of multiplicity one and is equal to $\mathrm{Sym}^m(\mathrm{Std})$, or $\mathrm{Sym}^m(\mathrm{Std}^{\vee})$, or $\bigwedge^2(\mathrm{Std})$, or the trivial representation $\mathbb 1$. These irreducible representations can be classified as follows: Let $k$ be a non-negative integer.
    \begin{enumerate}
        \item $\mathrm{Sym}^{4k}(\mathrm{Std})$ and $\mathrm{Sym}^{4k}(\mathrm{Std}^{\vee})$ contain the weight $0$, 
				as $0=k(f_1+f_2+f_3+f_4)$.
        \item $\mathrm{Sym}^{4k+1}(\mathrm{Std})$ and $\mathrm{Sym}^{4k+3}(\mathrm{Std}^{\vee})$ contain the weight $w_1:=f_1$, as $f_1=f_1+k(f_1+f_2+f_3+f_4)=-f_2-f_3-f_4-k(f_1+f_2+f_3+f_4)$.
        \item $\mathrm{Sym}^{4k+2}(\mathrm{Std})$, $\mathrm{Sym}^{4k+2}(\mathrm{Std}^{\vee})$ and $\bigwedge^2(\mathrm{Std})$ contain weight $w_2:=f_1+f_2$, as $f_1+f_2=f_1+f_2+k(f_1+f_2+f_3+f_4)=-f_3-f_4-k(f_1+f_2+f_3+f_4)$.
        \item $\mathrm{Sym}^{4k+3}(\mathrm{Std})$ and $\mathrm{Sym}^{4k+1}(\mathrm{Std}^{\vee})$ contain the weight $w_3:=f_1+f_2+f_3$, as $f_1+f_2+f_3=f_1+f_2+f_3+k(f_1+f_2+f_3+f_4)=-f_4-k(f_1+f_2+f_3+f_4)$.
    \end{enumerate}
		
		 Since $V$ is hypercubic, $w\in\Xi$ if and only if $-w\in \Xi$ and thus $V$ is self-dual. If $W$ is an irreducible factor of $V$ and $W$ is not self-dual,
         then $W\oplus W^\vee$ is a factor of $V$.
		Hence, if $\mathrm{Sym}^{4k+2}(\mathrm{Std})$ (or $\mathrm{Sym}^{4k+2}(\mathrm{Std}^{\vee})$)
		is an irreducible factor of $V$ then the weight $w_2$ above has multiplicity at least $2$  in $\Xi$,
		which is absurd.
		Similarly, $\mathrm{Sym}^{4k}(\mathrm{Std})$ (or $\mathrm{Sym}^{4k}(\mathrm{Std}^{\vee})$) for $k>0$
		cannot be an irreducible factor of $V$.
		By the facts that $V$ is self-dual and weight multiplicity-free, and the above classification,
 $V$ is one of the following cases, where $m\in\N$ is congruent to $1$ or $3$ modulo $4$.
    \begin{enumerate}
		\item $V=\mathrm{Sym}^{m}(\mathrm{Std}) \oplus \mathrm{Sym}^{m}(\mathrm{Std}^{\vee}) \oplus \bigwedge^2(\mathrm{Std})\oplus \mathbb 1$.
        \item $V=\mathrm{Sym}^{m}(\mathrm{Std}) \oplus \mathrm{Sym}^{m}(\mathrm{Std}^{\vee})\oplus \bigwedge^2(\mathrm{Std})$.
        \item $V=\mathrm{Sym}^{m}(\mathrm{Std}) \oplus \mathrm{Sym}^{m}(\mathrm{Std}^{\vee})\oplus \mathbb 1$.
        \item $V=\mathrm{Sym}^{m}(\mathrm{Std}) \oplus \mathrm{Sym}^{m}(\mathrm{Std}^{\vee})$.
        \item $V=\bigwedge^2(\mathrm{Std})\oplus \mathbb 1$.
        \item $V=\bigwedge^2(\mathrm{Std})$.
    \end{enumerate}
    Among these types of representations, the last two representations  $\bigwedge^2(\mathrm{Std})\oplus \mathbb 1$ and $\bigwedge^2(\mathrm{Std})$ are $7$- and $6$-dimensional, respectively, and thus are not hypercubic.
    
    Suppose that $V=\mathrm{Sym}^{m}(\mathrm{Std}) \oplus \mathrm{Sym}^{m}(\mathrm{Std}^{\vee}) \oplus U$ 
		is one of the first four cases. For $i=1,2,3$, let $L_i\subset\R^3$ be the affine line containing
    \begin{align}\label{pts}
        \{mf_4,(m-1)f_4+f_i,\ldots,f_4+(m-1)f_i,mf_i\}.
    \end{align}
		These are some weights coming from $\mathrm{Sym}^{m}(\mathrm{Std})$. Since \eqref{pts} is a subset of $\Xi$, we obtain 
		\begin{equation}\label{estimate}
		|L_i\cap \Xi|\geq (m+1)\hspace{.1in} \text{for} \hspace{.1in}i=1,2,3.
		\end{equation}
		Suppose $\Xi=Z_{d_1}\times Z_{d_2} \times Z_{d_3}\subset\R^3$ up to a change of coordinates.
		Since the three lines $L_1,L_2,L_3$ all pass through $mf_4$ and are not coplanar in $\R^3$, 
		the condition \ref{lem_lattice line lemma}(b) is satisfied and 
		we obtain $|\Xi|\geq (m+1)^3$ by \eqref{estimate} and the lemma.
	Hence, we obtain
    \begin{align*}
		\begin{split}
        \dim V=2\binom{m+3}{3}+\dim U=\frac{1}{3}(m+1)(m+2)(m+3)+\dim U\geq (m+1)^3,
				\end{split}
    \end{align*}
		equivalent to
		$$\frac{(m+1)(2m+3)(m-1)}{3}\leq \dim U.$$
		As $\dim U\in\{0,1,6,7\}$ and $m\in\N$ is congruent $1$ or $3$ modulo $4$,
		the only solution is $m=1$ which 
		implies $\dim V\in\{8,9,14,15\}$.
		Since $V$ is a faithful hypercubic representation of $A_3$, the dimension is given by $\dim V=\ell^3$ for some positive integer $\ell \geq 2$
		and the only possibility is $8=2^3$. 
In this case, $V=\mathrm{Std} \oplus \mathrm{Std}^{\vee}$ and the  formal character is given by
		$$\Xi=\{\pm f_1,\pm f_2, \pm f_3, \pm f_4\}=\{(\pm 1/2,\pm 1/2,\pm 1/2)\} \subset \mathbb R^3,$$
    which is hypercubic.
\end{proof}

\subsubsection{}
As $D_3$ and $A_3$ are isomorphic, it remains to classify faithful hypercubic representations $V$ of 
   $B_n$ $(n \geq 2)$, $C_n$ $(n \geq 3)$, and $D_n$ $(n \geq 4)$.
	
 Firstly, consider $B_n$ for $n\geq 2$. 
		Since $V$ is weight multiplicity-free, it follows from Theorem~\ref{thm_WMF}
		that 
		$$V=\delta_1\mathbb 1  \oplus \delta_2 \mathrm{Std}  \oplus \delta_3\mathrm{Spin}$$ 
		for $\delta_1,\delta_2,\delta_3 \in \{0,1\}$. If $n>2$, then we have
    \begin{align*}
        \dim(V) \leq \dim(\mathbb 1)+\dim(\mathrm{Std})+\dim(\mathrm{Spin})=2+2n+2^n<3^n,
    \end{align*}
    so $\dim(V)=2^n$. Consequently, either $V$ is $\mathrm{Spin}$, or $n=3$ and $V=\mathrm{Std}\oplus \mathbb 1$. But the latter one is not hypercubic (and in fact not even weight multiplicity-free), so $V=\mathrm{Spin}$ is the unique hypercubic representation in this case. If $n=2$, then there is an exceptional hypercubic representation $\mathrm{Std}\oplus \mathrm{Spin}$ (see Example \ref{example}), and one can readily derive that this exceptional representation and the spin representation are the only hypercubic representations of $B_2$.

    Secondly, consider $C_n$ for $n\geq 3$. If $n>3$, then 
		$V$ is either $\mathrm{Std}$ or $\mathrm{Std}\oplus\mathbb 1$ by Theorem~\ref{thm_WMF}, since the weights of $V$ are multiplicity-free. But both cases are not hypercubic.
When $n=3$, we obtain similarly that  
$$V=\delta_1\mathbb 1\oplus \delta_2\mathrm{Std} \oplus   \delta_3 U$$ 
for $\delta_1,\delta_2,\delta_3 \in \{0,1\}$, where $U=\bigwedge_{\mathrm{prim}}^3(\C^6)$ in Theorem \ref{thm_WMF}. 
As $\dim \mathbb 1 =1$, $\dim \mathrm{Std}=6$, and $\dim U=14$, it is impossible that $\dim V=\ell^3$ for any integer $\ell \geq 2$. Therefore, there is no hypercubic representations of $C_n$ for $n\geq 3$.

    Lastly, consider $D_n$ for $n \geq 4$. As above, one obtains $V=\delta_1\mathbb 1\oplus \delta_2\mathrm{Std}\oplus \delta_3\mathrm{Spin}^+\oplus \delta_4\mathrm{Spin}^-$ for $\delta_1,\delta_2,\delta_3,\delta_4 \in \{0,1\}$ and thus
    \begin{align*}
        \dim V \leq \dim(\mathbb 1)+\dim(\mathrm{Std})+\dim(\mathrm{Spin}^+)+\dim(\mathrm{Spin}^-)
				=1+2n+2^n<3^n.
    \end{align*}
  Therefore, we have $\dim(V)=2^n$. If $n \geq 5$, we have $(\delta_1,\delta_2,\delta_3,\delta_4)=(0,0,1,1)$ and hence $V=\mathrm{Spin}$ is the unique hypercubic representation of $D_n$. If $n=4$, we obtain three solutions
  \begin{align*}
      (\delta_1,\delta_2,\delta_3,\delta_4)=(0,1,1,0),(0,1,0,1),(0,0,1,1)
  \end{align*}
  which correspond to $\mathrm{Std}\oplus \mathrm{Spin}^+$, $\mathrm{Std}\oplus \mathrm{Spin}^-$, and 
  $\mathrm{Spin}$, respectively.

\subsection{Proofs of Theorem \ref{main1}(iv),(v)}\label{s2.9}
We handle the uniqueness assertions in Theorem \ref{main1}.
By the external tensor product \eqref{uni1},
the restriction $\psi|_{\Lieg_i}$ is some multiple of $\psi_i$.
It follows from 
 the classification in Theorem \ref{main1}(iii)(a)--(e) that $\psi_i$ is determined by $\psi|_{\Lieg_i}$.
Therefore, we obtain Theorem \ref{main1}(iv).
For Theorem \ref{main1}(v), suppose 
\begin{equation}\label{uni3}
(\Lieg_1,\psi_1)=\bigotimes_{j'=1}^{s'} (\Lieg_{1,j'},\psi_{1,j'})
\end{equation}
is another external tensor product of indecomposable faithful hypercubic representations,
where $\Lieg_1=\prod_{j'=1}^{s'} \Lieg_{1,j'}$ is some decomposition.
We now compare \eqref{uni2} with \eqref{uni3}.

\begin{lemma}\label{coincide}
The two sets of factors 
$F_1=\{\Lieg_{1,j}:~ 1\leq j\leq s\}$
and $F_1'=\{\Lieg_{1,j'}:~ 1\leq j'\leq s'\}$ of $\Lieg_1$ coincide.
\end{lemma}

\begin{proof}
Let $V_1$ be the ambient space of $\psi_1$ and $W_{1,j}$ be the maximal subspace of $V_1$ such that $\Lieg_{1,j}$ acts trivially.
Here, $\Lieg_{1,j}$ is either $A_1$ or $A_1\times A_1$ by Theorem \ref{main1}(ii).

Suppose first $\Lieg_{1,j}=A_1$. 
By the tensor product decomposition \eqref{uni2} and Theorem \ref{main1}(ii)(a),(b), we obtain
$\dim W_{1,j}<(\dim V_1)/2.$
If $\Lieg_{1,j}$ is a factor of some $\Lieg_{1,j'}=A_1\times A_1$,
then $\dim W_{1,j}=(\dim V_1)/2$ by \eqref{uni3} and Theorem \ref{main1}(ii)(c), which is absurd.
Hence, the simple factors in $F_1$ and $F_1'$ coincide.

Next, we suppose $\Lieg_{1,j}=A_1\times A_1$. In this case, we have $\dim W_{1,j}=0$ by the tensor product decomposition \eqref{uni2} and Theorem \ref{main1}(ii)(c).
If $\Lieg_{1,j}\notin F_1'$, then there exist two distinct non-simple factors $\Lieg_{1,j'}$ and $\Lieg_{1,j''}$
in $F_1'$ that intersect $\Lieg_{1,j}$ non-trivially, since the simple factors in $F_1$ and $F_1'$ coincide.
This implies that
$\dim W_{1,j}\neq 0$ by \eqref{uni3} and Theorem \ref{main1}(ii)(c), which is absurd.
Therefore, we conclude that $F_1$ and $F_1'$ coincide.
\end{proof}

Lemma \ref{coincide} induces 
a bijective correspondence $j\leftrightarrow j'$ so that 
$\Lieg_{1,j}=\Lieg_{1,j'}$ in $\Lieg_1$. Then Theorem \ref{main1}(ii) implies that
$$\text{some multiple of }\psi_{1,j}=\psi_1|_{\Lieg_{1,j}}=\psi_1|_{\Lieg_{1,j'}}=
\text{some multiple of }\psi_{1,j'}$$
and we obtain $\psi_{1,j}=\psi_{1,j'}$ from the classification in Theorem \ref{main1}(ii)(a)--(c).

\section{$\lambda$-independence of algebraic monodromy groups}

\subsection{Notation and terminology}\label{s3.1}
We collect the notation and terminology that will be used frequently.
\begin{itemize}
    \item $K$ (resp. $E$): a number field.
		\item $\Sigma_K$ (resp. $\Sigma_E$): the set of finite places of $K$ (resp. $E$).
    \item $v$ (resp. $\lambda$): an element of $\Sigma_K$ (resp. $\Sigma_E$).
    \item $p$ (resp. $\ell$): the residue characteristic of $v$ (resp. $\lambda$).
		\item $S_\ell$: the set of places in $\Sigma_K$ that divide the rational prime $\ell$.
		\item $\overline K$: an algebraic closure of $K$.
		\item $E_{\lambda}$ and $\F_\lambda$: the $\lambda$-adic completion of $E$ and the residue field of $E_\lambda$.
    \item $\overline E_{\lambda}$ (resp. $\overline \F_\lambda$): an algebraic closure of $E_{\lambda}$ (resp. $\F_\lambda$).
    \item $\Gal_K$: the absolute Galois group $\Gal(\overline K/K)$ equipped with the profinite topology.
    %\item $\mathrm{Frob}_v$: the Frobenius element at $v$ in $\Gal_K$,
		\item $\rho_\lambda:\Gal_K\to\GL_n(\overline E_\lambda)$: a \emph{$\lambda$-adic representation} of $K$, i.e., a continuous group homomorphism.
		\item $\bar\rho_\lambda^{\ss}:\Gal_K\to\GL_n(\overline \F_\lambda)$: the \emph{residue representation} of $\rho_\lambda$, i.e., 
		the semisimplification of the reduction modulo $\lambda$ of $\rho_\lambda$.
		\item $\bG_\lambda$ or $\bG_{\rho_\lambda}$: the \emph{algebraic monodromy group} of $\rho_\lambda$, i.e., the Zariski closure 
		of the Galois image $\rho_\lambda(\Gal_K)$ in $\GL_{n,\overline E_\lambda}$.
		\item $\iota_\lambda: \overline E_\lambda\to \C$: a field isomorphism for each $\lambda\in\Sigma_E$.
		\item $\bG_{\lambda,\C}$ (resp. $\bG_{\rho_\lambda,\C}$): the base change $\bG_\lambda\times_{\iota_\lambda}\C$ 
		(resp. $\bG_{\rho_\lambda}\times_{\iota_\lambda}\C$).
		\item $\bG^\circ$: the identity component of a linear algebraic group $\bG$.
		\item $\bG^{\der}$: the derived group $[\bG^\circ,\bG^\circ]$ of the identity component of $\bG$.
		\item $\epsilon_\ell$ (resp. $\bar\epsilon_\ell$): the $\ell$-adic (resp. mod $\ell$) cyclotomic character.
		\item A semisimple $\lambda$-adic representation $\sigma_\lambda$ is \emph{of type $A$} (resp. \emph{of type $A_1$}) 
if the semisimple part
$\mathrm{Lie}(\bG_{\sigma_\lambda,\C})^{\ss}$ has only type $A$ (resp. type $A_1$) factors.
This includes the case where the semisimple part is zero.
    \item $\mathbb{G}_m$ and $\mathbb{G}_a$: $\mathrm{Spec}(F[x,1/x])$ and $\mathrm{Spec}(F[x])$ for some field $F$.
    %\item $G^{\mathrm{der}}$: the derived subgroup of a reductive group $G$.
    %\item $\bG_{\rho}$: the algebraic monodromy group of a representation $\rho$.
    %\item $\bG_{\rho,F'}:=\bG_{\rho} \times_F F'$, where $\bG_{\rho}$ is the algebraic monodromy group of $\rho$ over a base field $F$ and $F'/F$ is a field extension.
\end{itemize}

\subsection{Compatible system of Galois representations}\label{s3.2}
We mainly follow \cite[$\mathsection5$]{BLGGT14} and \cite[$\mathsection1$]{PT15}.
%We introduce a notion of compatible system and some results which we will use. A continuous homomorphism
%\begin{align*}
  %  \rho_{\lambda}: \mathrm{Gal}_K \longrightarrow \mathrm{GL}_n(\overline E_{\lambda})
%\end{align*}
%is called a $\ell$-adic Galois representation of $K$ (we use the term `$\ell$-adic' instead of `$\lambda$-adic' due to the conventional reason).

\begin{defi}\label{csdef}
    A family of ($n$-dimensional) $\lambda$-adic Galois representations indexed by $\Sigma_E$,
    \begin{align}\label{abscs}
        \{\rho_{\lambda}: \Gal_K \longrightarrow \mathrm{GL}_n(\overline E_{\lambda})\}_{\lambda\in\Sigma_E},
    \end{align}
    is called a  {\it compatible system of $K$ defined over $E$} (or \emph{$E$-rational compatible system}) 
		if there exist a finite set of places $S \subset \Sigma_K$ and a polynomial $\Phi_v(T)\in E[T]$ for each $v\in\Sigma_K\backslash S$ such that
		the following conditions hold.
    \begin{enumerate}
        \item[(a)] For each $\lambda\in\Sigma_E$, the representation $\rho_{\lambda}$ is unramified 
				at every $v\in \Sigma_K\backslash (S\cup S_\ell)$. 
        \item[(b)] For each $\lambda\in\Sigma_E$ and $v\in \Sigma_K\backslash (S\cup S_\ell)$,  
				the characteristic polynomial of $\rho_{\lambda}(\mathrm{Frob}_v)$ satisfies
        \begin{align*}\label{Frobpoly}
            \det(\rho_{\lambda}(\mathrm{Frob}_v)-T\cdot \mathrm{id})=\Phi_v(T) \in E[T].
        \end{align*}
    \end{enumerate}
		The compatible system $\{\rho_\lambda\}$ is said to be \emph{semisimple} if each $\rho_\lambda$ is semisimple.
\end{defi}

\begin{remark}
When the $E$-rational 
compatible system \eqref{abscs} is semisimple, 
there is a finite extension $E'/E$ such that \eqref{abscs}
can be descended to a $\GL_n(E'_{\lambda'})$-valued 
$E$-rational compatible system $\{\rho_{\lambda'}:\Gal_K\to\GL_n(E'_{\lambda'})\}_{\lambda'\in\Sigma_{E'}}$ \cite{BH26}.
\end{remark}

\begin{defi}\label{wcsdef}
    A semisimple $E$-rational compatible system $\{\rho_{\lambda} : \mathrm{Gal}_{K} \to \mathrm{GL}_n(\overline{E}_{\lambda})\}_{\lambda\in\Sigma_E}$ 
		is called a \emph{strictly compatible system} if the following conditions are satisfied.
    \begin{enumerate}
        \item[(a)] For each $\lambda\in\Sigma_E$ and $v\in S_\ell$, 
the local representation $\rho_{\lambda}|_{\Gal_{K_v}}$ is de Rham and is further crystalline if $v \not\in S$.
        \item[(b)] For each embedding $\tau: K\hookrightarrow \overline E$, there is a multiset of $n$ integers $\mathrm{HT}_\tau$
				such that $\mathrm{HT}_{i\circ\tau}(\rho_{\lambda})=\mathrm{HT}_{\tau}$ 
				for any $\lambda$ and any $i:\overline E\hookrightarrow \overline E_\lambda$ over $E$.
    %\end{enumerate}
    %A weakly compatible system is called a \emph{strictly compatible system} if an extra condition holds.
    %\begin{enumerate}
        \item[(c)] For $\lambda\in\Sigma_E$ and $v\notin S_\ell$, 
				the semisimplified Weil-Deligne representation $\iota \mathrm{WD}(\rho_{\lambda}|_{\mathrm{Gal}_{K_v}})^{F-ss}$ is independent of $\lambda$
				and $\iota:\overline{E}_\lambda\stackrel{\simeq}{\rightarrow} \C$.
    \end{enumerate}
\end{defi}

\begin{defi}\label{regdef}
    A $\lambda$-adic representation $\rho_{\lambda} : \mathrm{Gal}_{K} \to \mathrm{GL}_n(\overline{E}_{\lambda})$ is said to be \emph{regular} if it is unramified almost everywhere, $\rho_\lambda$ is de Rham at any $v$ above $\ell$,
and $\mathrm{HT}_\tau(\rho_\lambda)$ consists of $n$ distinct numbers for any embedding $\tau: K \hto \overline E_\lambda$.

		A strictly compatible system $\{\rho_{\lambda}: \mathrm{Gal}_K \to \mathrm{GL}_n(\overline E_{\lambda})\}_{\lambda}$ is said to be \emph{regular} if some (and hence all) $\rho_{\lambda}$ is regular. 
\end{defi}

If $K$ is totally real or CM field and $\pi$ is a regular algebraic, polarized, cuspidal automorphic representation of $\mathrm{GL}_n(\mathbb A_K)$, then one can attach a regular, semisimple $E$-rational, strictly compatible system $\{\rho_{\pi,\lambda}: \mathrm{Gal}_K \to \mathrm{GL}_n(\overline E_{\lambda})\}_{\lambda}$ for some $E$ (cf. \cite[$\mathsection~2.1$]{BLGGT14} and the references therein). We call the Galois representation $\rho_{\pi,\lambda}$
(resp. compatible system  $\{\rho_{\pi,\lambda}\}$) \emph{automorphic}.

\subsection{Potential automorphy of Galois representations}\label{s3.3}
Given an $\ell$-adic Galois representation $\rho_{\ell}:\mathrm{Gal}_K \to \mathrm{GL}_n(\overline{\mathbb Q}_{\ell})$, 
one can ask whether $\rho_\ell$ is automorphic (and thus part of a compatible system). 
We shall rely on the following potential automorphy results when $K$ is totally real.

\begin{thm}{\cite[Theorem C]{BLGGT14}}\label{thm_BLGGT14 Thm C}
    Let $K$ be a totally real field, $\ell \geq 2(n+1)$ be a rational prime, and
    \begin{align*}
        \rho_{\ell}:\mathrm{Gal}_K \to \mathrm{GL}_n(\overline{\mathbb Q}_{\ell})
    \end{align*}
    be an $\ell$-adic Galois representation of $K$. %We denote by $\overline{\rho}^{\ss}_\ell$ the semisimplification of the reduction of $\rho_\ell$. 
		Suppose the following conditions are satisfied.
    \begin{enumerate}[(1)]
        \item (Unramified almost everywhere) $\rho_{\ell}$ is unramified at all but finitely many primes.
        \item (Odd essential self-duality) Either $\rho_{\ell}$ maps to $\mathrm{GSp}_n$ with totally odd multiplier or it maps to $\mathrm{GO}_n$ with totally even multiplier.
        \item (Potential diagonalizability and regularity) $\rho_{\ell}$ is potentially diagonalizable (and hence potentially crystalline) at each prime $v$ of $K$ above $\ell$ and for each embedding $\tau : K \to \overline{\mathbb Q}_{\ell}$ it has $n$ distinct  $\tau$-Hodge-Tate numbers.
        \item (Irreducibility) $\overline{\rho}_{\ell}^{\mathrm{ss}}|_{\mathrm{Gal}_{K(\zeta_{\ell})}}$ is irreducible, where $\zeta_{\ell}:=e^{2\pi i/\ell}$ is the primitive $\ell$th root of unity.
    \end{enumerate}
    Then we can find a finite Galois totally real extension $K'/K$ such that $\rho_{\ell}|_{\mathrm{Gal}_{K'}}$ is attached to a regular algebraic polarized cuspidal automorphic representation of $\mathrm{GL}_n(\mathbb A_{K'})$. Moreover, $\rho_{\ell}$ is part of a strictly compatible system of $K$.
\end{thm}

By combining Theorem \ref{thm_BLGGT14 Thm C} with some big image results (see $\mathsection\ref{s3.4}$),
we have the following potential automorphy result on certain regular three-dimensional subrepresentations 
in a ($\GL_n(E_\lambda)$-valued) strictly compatible system.

\begin{prop}{\cite[Proposition 2.12(b)]{Hui23b}}\label{prop_SO3 potential automorphy}
    Let $K$ be a totally real field and 
		$$\{\rho_{\lambda}:\mathrm{Gal}_{K} \to \mathrm{GL}_n(E_{\lambda})\}_{\lambda}$$ 
		be an $E$-rational strictly compatible system. 
		For almost all $\lambda\in\Sigma_E$, if $\sigma_{\lambda}$ is a regular three-dimensional subrepresentation of $\rho_{\lambda} \otimes \overline{E}_{\lambda}$ such that the derived subgroup $\bG_{\sigma_{\lambda}}^{\mathrm{der}}$ of its algebraic monodromy group is $\mathrm{SO}_3$ (as a group embedded in $\mathrm{GL_3}$), then there is a finite Galois totally real extension $K'/K$ such that $\sigma_{\lambda}|_{\mathrm{Gal}_{K'}}$ is attached to a regular algebraic polarized cuspidal automorphic representation of $\mathrm{GL}_3(\mathbb A_{K'})$. Moreover, such a subrepresentation $\sigma_{\lambda}$ is a part of a strictly compatible system of $K$.
\end{prop}

\subsection{Big image results of Galois subrepresentations}\label{s3.4}
We present some big image results 
for type $A$ Galois subrepresentations in 
certain ($\GL_n(E_\lambda)$-valued) semisimple $E$-rational compatible system.

\begin{thm}\cite[Theorem 1.2, Theorem 3.12(v)]{Hui23a}\label{big}
Let  $\{\rho_\lambda:\Gal_K\to\GL_n(E_\lambda)\}$
be a semisimple $E$-rational compatible system of a number field $K$. 
Suppose there exist some integers $N_1,N_2\geq 0$ and finite extension $K'/K$ such that the following conditions hold. 
\begin{enumerate}[(a)]
\item (Bounded tame inertia weights): for almost all $\lambda$ 
and each finite place $v$ of $K$ above $\ell$, 
the tame inertia weights of the local representation 
$(\bar\rho_{\lambda}^{\ss}\otimes\bar\epsilon_\ell^{N_1})|_{\Gal_{K_v}}$ belong to $[0,N_2]$.
\item (Potential semistability): for almost all $\lambda$ and each finite place $w$ of $K'$ not above $\ell$,
the semisimplification of the local representation $\bar\rho_{\lambda}^{\ss}|_{\Gal_{K_{w}'}}$ is unramified.
\end{enumerate}
For almost all $\lambda\in\Sigma_E$, the following assertions hold.
\begin{enumerate}[(i)]
\item  If $\sigma_\lambda$ is a type $A$ irreducible subrepresentation
of $\rho_\lambda\otimes\overline E_\lambda$, then the residual representation $\bar\sigma_\lambda^{\ss}$ is also irreducible.
\item If $\sigma_\lambda$ is a type $A$ Lie-irreducible subrepresentation
of $\rho_\lambda\otimes\overline E_\lambda$, then the restriction $\bar\sigma_\lambda^{\ss}|_{\Gal_{K^{ab}}}$ is also irreducible,
where $K^{ab}$ is the maximal abelian extension of $K$.
\end{enumerate}
\end{thm}

\begin{prop}\cite[Proposition 2.11]{Hui23b}\label{cond}
Let $\{\rho_\lambda:\Gal_K\to\GL_n(E_\lambda)\}$ be a ($E$-rational) 
strictly compatible system of a number field $K$.
Then $\{\rho_\lambda\}$ satisfies the conditions \ref{big}(a),(b) 
for some integers $N_1,N_2\geq 0$ and finite extension $K'/K$.
\end{prop}

\subsection{$\lambda$-independence of formal bi-character}\label{s3.5}
Let $F$ be an algebraically closed field of characteristic zero and $\bG\subset\GL_{n,F}$
a reductive subgroup. Denote by $\bT$ a maximal torus of $\bG$. 
Then  the intersection $\bT^{\ss}:=\bG^{\der}\cap \bT$ is a maximal torus of the semisimple group $\bG^{\der}$.
The following definition is independent of the choice of the maximal torus $\bT$.

\begin{defi}\label{formchar}
Let $\bG\subset\GL_{n,F}$ be a reductive subgroup.
    \begin{enumerate}[(1)]
\item The \emph{formal character} of $\bG$ is defined as the conjugacy class of the torus $\bT$ 
in $\mathrm{GL}_{n, F}$.
\item The \emph{formal bi-character} of $\bG$ is defined as the conjugacy class of 
the chain of subtori $\bT^{\ss}\subset\bT$ in $\mathrm{GL}_{n, F}$.
    \end{enumerate}
\end{defi}

Suppose that $\bG_1$ and $\bG_2$ are reductive subgroups of respectively 
$\mathrm{GL}_{n,F_1}$ and $\mathrm{GL}_{n,F_2}$, where $F_1$ and $F_2$ are algebraically closed fields of characteristic zero.
Embed $F_1$ and $F_2$ into an algebraically closed field $F$.
We say that the formal characters (resp. formal bi-characters) of $\bG_1$ and $\bG_2$ are \emph{the same}
if this is true for the base change $\bG_{1,F}$ and $\bG_{2,F}$ in $\GL_{n,F}$.
This definition is independent of the choice of the over-field $F$.

%\begin{defi}
 %   Let $\bG \subset \mathrm{GL}_{n,F}$ be a reductive subgroup. 
  %  \begin{enumerate}
      %  \item[\normalfont(i)] The \emph{formal character} of $\bG$ is the $\mathrm{GL}_{n,F}$-conjugacy class of a maximal torus $\bT$ of $\bG$.
      %  \item[\normalfont(ii)] The \emph{formal bi-character} of $\bG$ is the $\mathrm{GL}_{n,F}$-conjugacy class of a chain
      %  \begin{align*}
      %      \bT^{\mathrm{ss}}\subset \bT
      %  \end{align*}
     %   consisting of $\bT$ a maximal torus of $\bG$ and $\bT^{\mathrm{ss}}$ a maximal torus of $\bG^{\mathrm{der}}$.
   % \end{enumerate}
%\end{defi}

%We will frequently abbreviate the terminology so that the inclusion $T \subseteq \mathrm{GL}_{n,F}$ and $T^{\mathrm{ss}}\subseteq T \subseteq \mathrm{GL}_{n,F}$ with a fixed maximal torus $T$ of $G$ and $T^{\mathrm{ss}}$ of $G^{\mathrm{der}}$ are called the formal character and the formal bi-character of $G$, respectively. 

%\begin{remark}
  %  Let $F$ be an algebraically closed field of characteristic zero and consider a faithful representation $G^{\mathrm{der}} \hookrightarrow \mathrm{GL}_{n,F}$. In this case, the formal character of $G^{\mathrm{der}}$ is determined by the multiset of weights of this representation and vice versa, so it corresponds to the formal character of a Lie algebra representation $\mathfrak g^{\mathrm{ss}}=\mathrm{Lie}(G^{\mathrm{der}}) \to \mathrm{End}(F^n)$.
%\end{remark}

We have the following $\lambda$-independence results about  algebraic 
monodromy groups. 
%The following results state the important $\lambda$-independence regarding the algebraic monodromy group $\bG_{\lambda}$ of a semisimple compatible system $\{\rho_{\lambda}\}_{\lambda}$.   

\begin{theorem}\label{thm_formal bi-char indep}
Let $K$ be a number field and $\{\rho_\lambda:\Gal_K\to\GL_n(\overline E_\lambda)\}$ be a semisimple $E$-rational compatible system 
with algebraic monodromy groups $\{\bG_\lambda\}$.
    \begin{enumerate}
        \item[(i)] \cite{Ser81} The component group $\bG_{\lambda}/\bG_{\lambda}^{\circ}$ is independent of $\lambda$. In particular, 
				the connectedness of $\bG_\lambda$ is independent of $\lambda$.
        \item[(ii)] \cite{Ser81} The formal character of $\bG_{\lambda} \subset \mathrm{GL}_{n,\overline E_\lambda}$ is independent of $\lambda$. In particular, the rank of $\bG_{\lambda}$ is independent of $\lambda$. 
				\item[(iii)] \cite[Theorem 3.19, Remark 3.22]{Hui13} The formal bi-character of $\bG_{\lambda} \subset \mathrm{GL}_{n,\overline E_\lambda}$ is independent of $\lambda$. In particular, the formal character of $\bG_{\lambda}^{\der}\subset\GL_{n,\overline E_\lambda}$ 
				(resp. semisimple rank of $\bG_\lambda$) is independent of $\lambda$.
    \end{enumerate}
\end{theorem}

%\begin{theorem}\cite[Theorem 3.19, Remark 3.22]{Hui13}\label{thm_formal bi-char indep}
%Let $K$ be a number field and $\{\rho_\lambda:\Gal_K\to\GL_n(\overline E_\lambda)\}$ be a semisimple $E$-rational compatible system 
%with algebraic monodromy groups $\{\bG_\lambda\}$.
  % The formal bi-character of $\bG_{\lambda} \subset \mathrm{GL}_{n,\overline E_\lambda}$ is independent of $\lambda$. In particular, the rank and the semisimple rank of $\bG_{\lambda}$ are both independent of $\lambda$.
%\end{theorem}

\begin{remark}\label{remconn}
For a compatible system $\{\rho_\lambda\}$ with algebraic monodromy groups $\{\bG_\lambda\}$,
the semisimplification $\{\rho_\lambda^{\ss}\}$ is a semisimple compatible system with algebraic monodromy groups $\{\bG_\lambda/\bU_\lambda\}$,
where $\bU_\lambda$ is the unipotent radical of $\bG_\lambda$. Since $\bU_\lambda$ is always connected,
it follows that $\bG_\lambda/\bG_\lambda^\circ$ is independent of $\lambda$ by Theorem \ref{thm_formal bi-char indep}(i).
\end{remark}

A compatible system $\{\rho_\lambda\}$ 
is said to be \emph{connected} if the algebraic monodromy group $\bG_\lambda$ is connected for some $\lambda$, equivalently, 
for all $\lambda$ by Remark \ref{remconn}. For any compatible system $\{\rho_\lambda\}$ (of $K$), there exists
a finite extension $L/K$ such that the restriction $\{\rho_\lambda|_{\Gal_L}\}$ is connected.

\subsection{A refinement of Theorem \ref{thm_formal bi-char indep}(ii),(iii)}\label{s3.6}
In this subsection, we refine Theorem \ref{thm_formal bi-char indep}(ii),(iii) 
by utilizing the method in \cite[$\mathsection3$]{Hui13}.
Suppose that for each $1\leq i \leq k$, $\{\rho_\lambda^{(i)}:\Gal_K\to\GL_{n_i}(\overline E_\lambda)\}$ 
is a semisimple $E$-rational compatible system of $K$.
Consider the semisimple $E$-rational compatible system
\begin{equation}\label{bigcs}
\{\hat\rho_\lambda:=\bigoplus_{i=1}^k \rho_\lambda^{(i)}:\Gal_K\to \prod_{i=1}^k \GL_{n_i}(\overline E_\lambda)\}
\end{equation}
given by direct sum and let 
$$\{\hat\bG_\lambda\subset \prod_{i=1}^k \GL_{n_i,\overline E_\lambda}\}$$
be the system of algebraic monodromy groups. For all $\lambda\in\Sigma_E$, let
$$\hat\bT_\lambda^{\ss}\subset\hat\bT_\lambda$$ 
be a chain of subtori in $\prod_{i=1}^k \GL_{n_i,\overline E_\lambda}$ 
where $\hat\bT_\lambda$ is a maximal torus of $\hat\bG_\lambda$
and $\hat\bT_\lambda^{\ss}$ is a maximal torus of $\hat\bG_\lambda^{\der}$.

\begin{thm}\label{refine} 
After the base change $\iota_\lambda:\overline E_\lambda\stackrel{\simeq}{\rightarrow}\C$ for all $\lambda\in\Sigma_E$,
\begin{enumerate}[(i)]
\item the conjugacy class of the subtorus $\hat\bT_{\lambda,\C}$ in $\prod_{i=1}^k \GL_{n_i,\C}$ is independent of $\lambda$;
\item the conjugacy class of the chain $\hat\bT_{\lambda,\C}^{\ss}\subset\hat\bT_{\lambda,\C}$ in 
$\prod_{i=1}^k \GL_{n_i,\C}$ is independent of $\lambda$.
\end{enumerate}
\end{thm}

\begin{proof}
We may assume that \eqref{bigcs} is connected by taking a finite extension of $K$
and the chain $\hat\bT_{\lambda,\C}^{\ss}\subset\hat\bT_{\lambda,\C}$
is contained in the diagonal $\prod_{i=1}^k \mathbb{G}_m^{n_i}$ of $\prod_{i=1}^k \GL_{n_i,\C}$
by conjugation.

(i). The idea (as in \cite{Ser81}) is to look at the morphism 
$$\mathrm{Char}:=\prod_{i=1}^k \mathrm{Char}^{(i)}: \prod_{i=1}^k \GL_{n_i,\C}\to \prod_{i=1}^k (\mathbb{G}_a^{n_i-1}\times\mathbb{G}_m)$$
where $\mathrm{Char}^{(i)}(A)=(a_1,...,a_{n_i})$ if $A\in \GL_{n_i}(\C)$ 
and $\det(A-T\cdot \mathrm{id})=\sum_{j=0}^{n_i} a_j T^{n_i-j}$.
By the compatibility conditions \ref{csdef}(a),(b) and Chebotarev's density theorem, 
$\mathrm{Char}(\hat\bG_{\lambda,\C})$ is independent of $\lambda$.
Since $\hat\bT_{\lambda,\C}$ is a maximal torus of the connected $\hat\bG_{\lambda,\C}$,
it follows that $\mathrm{Char}(\hat\bT_{\lambda,\C})=\mathrm{Char}(\hat\bG_{\lambda,\C})$
is independent of $\lambda$. Together with the facts that 
$\hat\bT_{\lambda,\C}\subset \prod_{i=1}^k \mathbb{G}_m^{n_i}$ (the diagonal) 
is a closed irreducible subvariety
and the restriction $\mathrm{Char}|_{\prod_{i=1}^k \mathbb{G}_m^{n_i}}$ is a finite morphism,
we obtain assertion (i).

(ii). This part follows closely the idea behind \cite[Theorem 3.19]{Hui13}. 
Let $\mathfrak{m}$ be a modulus of $K$ and $S_\mathfrak{m}$ be the 
Serre group\footnote{The Serre group $S_\mathfrak{m}$ is a diagonalizable group defined over $\Q$ and its dimension depends
only on $K$, i.e., independent of $\mathfrak m$.} of $K$ with respect $\mathfrak{m}$. 
Associated to $S_\mathfrak{m}$ a family of abelian $\ell$-adic representations
$$\{\alpha_\ell:\Gal_K\to S_\mathfrak{m}(\Q_\ell)\}_{\ell\in\Sigma_\Q}$$
such that for any $E$-morphism $\phi:S_{\mathfrak{m},E}\to \GL_{m,E}$, the composition 
$$\{\phi_\lambda: \Gal_K\stackrel{\alpha_\ell}{\rightarrow}S_\mathfrak{m}(E_\lambda)
\stackrel{\phi\otimes_E E_\lambda}{\longrightarrow}\GL_m(E_\lambda)\subset\GL_m(\overline E_\lambda)\}_{\lambda\in\Sigma_E}$$
is an abelian semisimple $E$-rational compatible system \cite[Chapter II]{Ser98}.

Assume $\phi$ is now faithful. 
Consider the semisimple compatible system $\{\hat\rho_\lambda\oplus\phi_\lambda\}$ (a direct 
sum of $k+1$ compatible systems) and let 
$$\{\widetilde\bG_\lambda\subset (\prod_{i=1}^k \GL_{n_i,\overline E_\lambda})\times\GL_{m,\overline E_\lambda}\}$$
be the system of algebraic monodromy groups. 
For all $\lambda$, let $\widetilde\bT_\lambda$ 
%be a subtorus in $(\prod_{i=1}^k \GL_{n_i,\overline E_\lambda})\times\GL_{m,\overline E_\lambda}$ such that $\widetilde\bT_\lambda$ is 
a maximal torus of $\widetilde\bG_\lambda$.
It follows from assertion (i) that 
\begin{equation}\label{words}
\text{the conjugacy class of $\widetilde\bT_{\lambda,\C}$
in $(\prod_{i=1}^k \GL_{n_i,\C})\times\GL_{m,\C}$ is independent of $\lambda$}.
\end{equation}
Let $\pi_1$ (resp. $\pi_2$) be the projection map of $(\prod_{i=1}^k \GL_{n_i,\C})\times\GL_{m,\C}$
to the first $k$ factors (resp. to the last factor).
According to \cite[Remark 3.22]{Hui13} and the proof of \cite[Theorem 3.19]{Hui13}, the chain
\begin{equation}\label{refinechain}
\mathrm{Ker}(\pi_2|_{\widetilde\bT_{\lambda,\C}})^\circ\subset \pi_1(\widetilde\bT_{\lambda,\C})
\end{equation}
is the formal bi-character of $\pi_1(\widetilde\bG_{\lambda,\C})=\hat\bG_{\lambda,\C}$ for all $\lambda$.
By \eqref{words}, the conjugacy class of \eqref{refinechain} in $\prod_{i=1}^k \GL_{n_i,\C}$
is independent of $\lambda$.
\end{proof}

Suppose $k=2$ in \eqref{bigcs} and denote by $\pi_1:\GL_{n_1,\overline E_\lambda}\times\GL_{n_2,\overline E_\lambda}\to \GL_{n_1,\overline E_\lambda}$
the projection to the first factor.

\begin{cor}\label{oneiso}
If $\pi_1:\hat\bG_{\lambda_0}\to \pi_1(\hat\bG_{\lambda_0})$ is an isomorphism for some $\lambda_0$, then the surjection
\begin{equation}\label{proj1}
\pi_1:\hat\bG_{\lambda}\to \pi_1(\hat\bG_{\lambda})
\end{equation}
 is an isomorphism for all $\lambda$.
\end{cor}

\begin{proof}
First, note that $\pi_1(\hat\bG_{\lambda})$ is the algebraic monodromy group of $\rho_\lambda^{(1)}$ for all $\lambda$.
Then the condition that \eqref{proj1} is an isomorphism at $\lambda_0$ and 
Theorem \ref{thm_formal bi-char indep}(i),(ii) imply that the ranks (resp. numbers of components) of $\hat\bG_{\lambda}$ 
and $\pi_1(\hat\bG_{\lambda})$ are equal for each $\lambda$.
Thus, the kernel $\bC_\lambda$ of the surjection \eqref{proj1} is a finite normal subgroup of $\hat\bG_{\lambda}^\circ$ for each $\lambda$.
We obtain that $\bC_\lambda\subset \hat\bT_{\lambda}$ (a maximal torus of $\hat\bG_{\lambda}$) for each $\lambda$.
Since the restriction of \eqref{proj1} to $\hat\bT_{\lambda}$ is injective at $\lambda=\lambda_0$, 
this is also true for all $\lambda$ by Theorem \ref{refine}(i).
Therefore,  $\bC_\lambda$ is trivial for all $\lambda$.
\end{proof}

\subsection{Invariance of roots and $\lambda$-independence}\label{s3.7}
Suppose $\{\rho_\lambda:\Gal_K\to\GL_n(\overline E_\lambda)\}$ is 
a connected semisimple $E$-rational compatible system with algebraic monodromy groups $\{\bG_\lambda\}$. 
After the base change $\iota_\lambda:\overline E_\lambda\stackrel{\simeq}{\rightarrow}\C$
for all $\lambda\in\Sigma_E$, we obtain a system of complex connected
reductive subgroups 
$$\{\bG_{\lambda,\C}\subset\GL_{n,\C}\}.$$
Up to $\GL_{n,\C}$-conjugation, there exists a chain of subtori
\begin{equation}\label{commonchain}
\bT_\C^{\ss}\subset\bT_\C
\end{equation}
in $\GL_{n,\C}$ such that $\bT_\C$ is a maximal torus of $\bG_{\lambda,\C}$ for all $\lambda$
and $\bT_\C^{\ss}$ is a maximal torus of $\bG_{\lambda,\C}^{\der}$ for all $\lambda$
by Theorem \ref{thm_formal bi-char indep}(iii). 
We give a crucial criterion, called the \emph{invariance of roots}, 
for the connected reductive subgroups $\{\bG_{\lambda,\C}\}$ to be conjugate in $\GL_{n,\C}$.
Denote by $\X(\bT_\C^{\ss}):=\mathrm{Hom}(\bT_\C^{\ss},\mathbb{G}_m)$ the character group of $\bT_\C^{\ss}$.

\begin{prop}\label{invroot}\cite[Corollary 3.9]{Hui18}
For two primes $\lambda_1$ and $\lambda_2\in \Sigma_E$,
let $\Phi_1\subset \X(\bT_\C^{\ss})$ and $\Phi_2\subset \X(\bT_\C^{\ss})$ be the root systems
of the semisimple groups $\bG_{\lambda_1,\C}^{\der}$ and $\bG_{\lambda_2,\C}^{\der}$
with respect to the common maximal torus $\bT_\C^{\ss}$ in \eqref{commonchain}.
If $\Phi_1=\Phi_2$, then $\bG_{\lambda_1,\C}$ and $\bG_{\lambda_2,\C}$
are conjugate in $\GL_{n,\C}$. 
\end{prop}

The main strategy for Theorem \ref{thm_main2}
is to construct some auxiliary compatible system $\{\phi_\lambda\}$
so that the invariance of roots criterion is fulfilled.

\begin{prop}\label{strategy}
Let $\{\rho_\lambda:\Gal_K\to\GL_n(\overline E_\lambda)\}$ be
a connected semisimple $E$-rational compatible system with algebraic monodromy groups $\{\bG_\lambda\subset\GL_{n,\overline E_\lambda}\}$. 
Let $\{\phi_\lambda:\Gal_K\to\GL_m(\overline E_\lambda)\}$ be a semisimple $E$-rational compatible system 
and consider the compatible system $\{\rho_\lambda\oplus\phi_\lambda\}$ with 
algebraic monodromy groups $\{\hat\bG_\lambda\subset\GL_{n,\overline E_\lambda}\times \GL_{m,\overline E_\lambda}\}$.
Let $\pi_1$ (resp. $\pi_2$) be the projection map of $\GL_{n,\overline E_\lambda}\times \GL_{m,\overline E_\lambda}$
to the first (resp. second) factor.
Suppose the following conditions hold.
\begin{enumerate}[(a)]
\item For all $\lambda$ (equivalently, for some $\lambda$), the first projection $\pi_1: \hat\bG_\lambda\to \bG_\lambda$ is an isomorphism.
\item For all $\lambda$, the composition $\bG_\lambda\stackrel{\pi_1^{-1}}{\longrightarrow}\hat\bG_\lambda\stackrel{\pi_2}{\rightarrow}\GL_{m,\overline E_\lambda}$ is the adjoint representation on the semisimple part $\mathrm{Lie}(\bG_\lambda^{\der})$,
where $\pi_1^{-1}$ means the inverse of the isomorphism in (a).
\end{enumerate}
After the base change $\overline E_\lambda\simeq \C$, 
the conjugacy class of $\bG_{\lambda,\C}$ in $\GL_{n,\C}$ is independent of $\lambda$.
\end{prop}

\begin{proof}
In the condition (a), Corollary \ref{oneiso} implies that the assertion
``$\pi_1: \hat\bG_\lambda\to \bG_\lambda$ is an isomorphism'' is independent of $\lambda$.

Up to $\GL_{n,\C}\times \GL_{m,\C}$-conjugation, there exists a chain of subtori
$$\hat\bT_{\C}^{\ss}\subset\hat\bT_{\C}$$
in $\GL_{n,\C}\times \GL_{m,\C}$ such that $\hat\bT_{\C}$ (resp. $\hat\bT_{\C}^{\ss}$)
is a maximal torus of $\hat\bG_{\lambda,\C}$ (resp. $\hat\bG_{\lambda,\C}^{\der}$) for all $\lambda$ by Theorem \ref{refine}(ii). 
Then we may take the chain \eqref{commonchain} as
$$\pi_1(\hat\bT_{\C}^{\ss})\subset\pi_1(\hat\bT_{\C}).$$
By (b), the root system of $\bG_{\lambda,\C}^{\der}$ with respect to $\pi_1(\hat\bT_{\C}^{\ss})$
is given by the weights of
$$\pi_1(\hat\bT_{\C}^{\ss})\stackrel{\pi_1^{-1}}{\longrightarrow} \hat\bT_{\C}^{\ss}\stackrel{\pi_2}{\longrightarrow}\GL_{m,\C},$$
which is independent of $\lambda$. Therefore, the conjugacy class of $\bG_{\lambda,\C}$ in $\GL_{n,\C}$ is independent of $\lambda$ 
by Proposition \ref{invroot}.
\end{proof}

\subsection{Hodge-Tate lift}\label{s3.8}
Let $F$ be a finite field extension of $\Q_\ell$ and $\bH$ a linear algebraic group defined over 
$\overline{\mathbb Q}_{\ell}$ in this subsection. 
An $\ell$-adic representation $\rho: \mathrm{Gal}_{F} \to \bH(\overline{\mathbb Q}_{\ell})$ is said to be \emph{Hodge-Tate} 
if for any representation $\bH \to \mathrm{GL}_{k,\overline{\mathbb Q}_{\ell}}$, the composition
$$\mathrm{Gal}_{F} \stackrel{\rho}{\rightarrow} \bH(\overline{\mathbb Q}_{\ell}) \to \mathrm{GL}_k(\overline{\mathbb Q}_{\ell})$$ 
is a Hodge-Tate representation.

\begin{prop}{\cite[Corollary 3.2.12]{Pat19}}\label{prop_Hodge-Tate lift}
    Let $\bH' \rightarrow \bH$ be a surjection of $\overline\Q_\ell$-linear algebraic groups whose kernel is a central torus
		and $\rho: \mathrm{Gal}_{F} \to \bH(\overline{\mathbb Q}_{\ell})$ be a Hodge-Tate representation of $F$. Then there exists a Hodge-Tate lift $\tilde{\rho}:\mathrm{Gal}_{F} \to \bH'(\overline{\mathbb Q}_{\ell})$ so that the following diagram commutes.
    \begin{equation*}\label{liftdiagram}
		\begin{tikzcd}
	& {\bH'(\overline{\mathbb Q}_{\ell})} \\
	{\mathrm{Gal}_{F}} & {\bH(\overline{\mathbb Q}_{\ell})}
	\arrow[two heads, from=1-2, to=2-2]
	\arrow["{\tilde{\rho}}", dashed, from=2-1, to=1-2]
	\arrow["\rho"', from=2-1, to=2-2]
\end{tikzcd}
\end{equation*}
\end{prop}

\begin{prop}\label{prop_GL2r lift}
    Let $\rho_0: \mathrm{Gal}_{F} \to \GL_k(\overline \Q_\ell)$ be an $\ell$-adic Hodge-Tate representation
		such that the $\tau$-Hodge-Tate numbers are distinct for any embedding $\tau: F\to\overline\Q_\ell$.
		Suppose the following conditions hold.
		\begin{enumerate}[(a)]
		\item $\rho_0$ factors through a connected reductive subgroup $\bH\subset\GL_{k,\overline\Q_\ell}$ of semisimple rank $r_0\in\N$.
		\item $\bH$ is irreducible on the ambient space.
		\item $\bH$ contains the homothety $\mathbb{G}_m$ of $\GL_{k,\overline\Q_\ell}$.
		\item $\bH$ is of type $A_1$, i.e., $\mathrm{Lie}(\bH^{\der})$ has only $A_1$-factors.
		\end{enumerate}
Then there exist a surjective morphism $\bH':=\prod_{j=1}^{r_0}\GL_{2,\overline\Q_\ell}\to \bH$ whose kernel is a central torus
and a Hodge-Tate lifting 
    \begin{align*}
       \tilde{\rho_0}:=\bigoplus_{j=1}^{r_0} (f_j: \mathrm{Gal}_{F} \to \mathrm{GL}_2(\overline \Q_\ell))
    \end{align*}
    such that $\tau$-Hodge-Tate numbers of $f_j$ are distinct for every $1\leq j\leq r_0$.
\end{prop}

\begin{proof}
    This is essentially the same as the proof of \cite[Lemma~4.9]{HL24} (using Proposition \ref{prop_Hodge-Tate lift}). 	
		% The lift given in \cite[Lemma~4.9]{HL24} is the local Hodge-Tate lift, but we can obtain the global crystalline lift in our case using Proposition~\ref{prop_crystalline lift}.
\end{proof}

\subsection{Proof of Theorem~\ref{thm_main2}}\label{s3.9}
\subsubsection{}
Now $K$ is totally real. By the condition ~\ref{thm_main2}(a)
 and \cite[Lemma 5.3.1]{BLGGT14}, we may assume $E$ is large enough such that
$\rho_\lambda$ is $\GL_n(E_\lambda)$-valued for all $\lambda$.
Then the algebraic monodromy group $\bG_\lambda$ is defined over $E_\lambda$ for all $\lambda$. 
To prove Theorem ~\ref{thm_main2}, it suffices (by Proposition \ref{strategy}) 
to construct some compatible system $\{\phi_\lambda\}$ (see \eqref{aux}) so that the conditions \ref{strategy}(a),(b) hold (enlarging $E$ if necessary).
This will be achieved by the potential automorphy result, Proposition \ref{prop_SO3 potential automorphy}.

Let $r\in\N$ be the rank of $\bG_\lambda^{\der}$ (independent of $\lambda$ by Theorem \ref{thm_formal bi-char indep}(iii)).
The conditions \ref{thm_main2}(c),(d) and Corollary \ref{cor_L at most 2 at most 3}(i)
imply that the semisimple Lie algebra $\mathrm{Lie}(\bG_{\lambda,\overline E_\lambda}^{\der})$ is 
of type $A_1$ for all $\lambda$. 
Since the $3r$-dimensional $\mathrm{Lie}(\bG_\lambda^{\der})$ is  
acted on by $\bG_\lambda$ (via adjoint action) and also by $\Gal_K$, we obtain a semisimple $\lambda$-adic representation
\begin{equation*}\label{adjoint}
\psi_\lambda:\Gal_K\to\GL_{3r}(E_\lambda)
\end{equation*}
such that $\psi_\lambda\otimes \overline E_\lambda$ is the direct sum of 
$r$ representations $\psi_{\lambda,i}:\Gal_K\to\GL_3(\overline E_\lambda)$ for $1\leq i\leq r$,
where each $\psi_{\lambda,i}$ corresponds to a simple ($A_1$-) factor of $\mathrm{Lie}(\bG_{\lambda,\overline E_\lambda}^{\der})$.
Since $\bG_\lambda$ is connected by the condition \ref{thm_main2}(b), the algebraic monodromy group $\bG_{\psi_{\lambda,i}}$
 (resp. $\bG_{\psi_\lambda\otimes\overline E_\lambda}$)
is isomorphic to $\SO_{3}$ (resp. $\SO_{3}^r$) for all $i$.

\begin{prop}\label{propregular}
For each $1\leq i\leq r$, the three dimensional representation $\psi_{\lambda,i}$ is regular.
\end{prop}

\begin{proof}
Since $\{\rho_\lambda\}$
is strictly compatible by the condition \ref{thm_main2}(a), the system $\{\rho_\lambda\otimes\rho_\lambda^\vee\}$
is also strictly compatible. 
As $\psi_\lambda$ is a subrepresentation of $\rho_\lambda\otimes\rho_\lambda^\vee$,
it is de Rham at places $v$ above $\ell$.
It remains to show that $\psi_{\lambda,i}|_{\Gal_{K_v}}$ has distinct
$\tau$-Hodge-Tate numbers for any embedding $\tau: K_v\to \overline E_\lambda$ (see Definition \ref{regdef}).

Take $F:=K_v$ and $\overline E_\lambda\simeq \overline \Q_\ell$ a field isomorphism.
For each non-abelian irreducible subrepresentation $\rho_0$ of $\rho_\lambda\otimes\overline E_\lambda$,
the restriction 
\begin{equation}\label{restrictF}
\rho_0:\Gal_F\to\GL_k(\overline \Q_\ell)
\end{equation}
has distinct $\tau$-Hodge-Tate numbers 
by the condition \ref{thm_main2}(a). Moreover, \eqref{restrictF} 
satisfies the conditions \ref{prop_GL2r lift}(a)--(d) with $\bH:=\bG_{\rho_0}\mathbb{G}_m$, where $\mathbb{G}_m$
denotes the homothety of $\GL_{k,\overline\Q_\ell}$. Hence, we obtain $f_1,...,f_{r_0}$ by Proposition \ref{prop_GL2r lift},
where $r_0\in\N$ is the semisimple rank of $\bG_{\rho_0}$.

Since $\psi_{\lambda,i}$ is a simple $A_1$-factor of $\psi_\lambda\otimes\overline\Q_\ell=\mathrm{Lie}(\bG_{\lambda,\overline\Q_\ell}^{\der})$, 
there exists an irreducible $\rho_0$ (as above) such that $\psi_{\lambda,i}$ is injective via the natural $\Gal_F$-equivariant morphism
$$\mathrm{Lie}(\bG_{\lambda,\overline\Q_\ell}^{\der})\to\mathrm{Lie}(\bG_{\rho_0}^{\der})=\mathrm{Lie}(\bH^{\der})\simeq \mathrm{Lie}(\bH'^{\der}),$$
where the first arrow comes from $\rho_0$-projection and the last isomorphism holds because the kernel of $\bH'\twoheadrightarrow\bH$ 
is a central torus by Proposition \ref{prop_GL2r lift}. Hence, we find some $f_j:\Gal_F\to\GL_2(\overline\Q_\ell)$ 
such that the restriction
\begin{equation}\label{ad0}
\psi_{\lambda,i}|_{\Gal_{F}}\simeq \mathrm{ad}^0(f_j)
\end{equation}
as $\Gal_F$-representations, where $\mathrm{ad}^0(f_j)$ denotes
the trace zero part of $f_j\otimes f_j^\vee=\End(f_j)$.
Since the $\tau$-Hodge-Tate numbers of $f_j$ are distinct by Proposition \ref{prop_GL2r lift},
the same holds for $\psi_{\lambda,i}|_{\Gal_{F}}$ by \eqref{ad0}.
\end{proof}

\subsubsection{}
Now $\psi_{\lambda,i}$ is regular and $\bG_{\psi_{\lambda,i}}=\SO_3\subset\GL_3$ for all $\lambda\in\Sigma_E$ and $1\leq i\leq r$.
Then Proposition \ref{prop_SO3 potential automorphy} 
(on the $E$-rational strictly compatible system $\{\rho_\lambda\otimes\rho_\lambda^\vee\}$) 
implies the existence of a prime $\lambda_0\in\Sigma_E$ such that $\psi_{\lambda_0,i}$ 
is potentially automorphic and is part of a semisimple $E$-rational strictly compatible system 
$\{\phi_{i,\lambda}\}$ for all $1\leq i\leq r$ (by enlarging $E$ if necessary).
Therefore,  we obtain a $3r$-dimensional semisimple $E$-rational compatible system
\begin{equation}\label{aux}
\{\phi_\lambda:=\bigoplus_{i=1}^r \phi_{i,\lambda}:\Gal_K\to \prod_{i=1}^r \GL_3(\overline E_\lambda)\}_{\lambda\in\Sigma_E}
\end{equation}
of the totally real $K$ such that $\phi_{\lambda_0}=\psi_{\lambda_0}$. 

\begin{prop}\label{auxgp}
For all $\lambda$, the algebraic monodromy group $\bG_{\phi_{\lambda}}$ of $\phi_\lambda$ is $\SO_3^r\subset \GL_3^r$.
\end{prop}

\begin{proof}
Note that the assertion holds at $\lambda_0$. In particular,
the algebraic monodromy group of $\phi_{i,\lambda_0}$ satisfies
\begin{equation}\label{so3}
\bG_{\phi_{i,\lambda_0}}=\SO_3\subset\GL_3
\end{equation} 
for every $i$.
As Proposition \ref{prop_SO3 potential automorphy} asserts that 
the restriction 
$\{\phi_{i,\lambda}|_{\Gal_{K_i}}\}$ is automorphic for 
some totally real extension $K_i/K$, 
the algebraic monodromy group of $\phi_{i,\lambda}|_{\Gal_{K_i}}$
is independent of $\lambda$ (see \cite[Proposition 4.16(iii)]{Hui23a}),
which must be $\SO_3\subset\GL_3$ by \eqref{so3}. By \eqref{so3} again and Theorem \ref{thm_formal bi-char indep},
we obtain
\begin{equation*}\label{so3'}
\bG_{\phi_{i,\lambda}}=\SO_3\subset\GL_3
\end{equation*} 
for all $\lambda$ and $i$, which implies that 
\begin{equation}\label{so3''}
\bG_{\phi_{\lambda}}\subset \SO_3^r
\end{equation} 
in $\GL_3^r$ for all $\lambda$. Since \eqref{so3''} is an equality at $\lambda_0$,
it follows from Theorem \ref{thm_formal bi-char indep} that \eqref{so3''} is an equality for all $\lambda$.
We are done.
\end{proof}

\subsubsection{}
With $\{\phi_\lambda\}$ at hand, we are in the setting of Proposition \ref{strategy}.
To complete the proof of Theorem \ref{thm_main2}, 
it remains to verify the conditions \ref{strategy}(a),(b).
The algebraic monodromy group $\hat\bG_\lambda$ of $\rho_\lambda\oplus\phi_\lambda$ 
is contained in the product $\bG_\lambda\times\bG_{\phi_\lambda}\subset\GL_n\times \GL_{3r}$. 
Let $\pi_1$ and $\pi_2$ be respectively, the projection to the first and second factor.
Since $\phi_{\lambda_0}$ is given by the adjoint action on $\mathrm{Lie}(\bG_{\lambda_0}^{\der})$, the
condition \ref{strategy}(a) holds (at $\lambda=\lambda_0$). 
%Hence, the ranks (resp. semisimple ranks) of $\bG_\lambda$ and $\hat\bG_\lambda$ are equal for all $\lambda$ by Theorem \ref{thm_formal bi-char indep}. It follows that the kernel 
%\begin{equation}\label{kernelpi1}
%\mathrm{Ker}(\pi_1:\hat\bG_\lambda\to\bG_\lambda)
%\end{equation}
%is finite and normal in $\hat\bG_\lambda$. Since $\pi_2:\hat\bG_\lambda\to\bG_{\phi_\lambda} $ is surjective, \eqref{kernelpi1} is a finite normal subgroup of $\{1\}\times \bG_{\phi_\lambda}$. Since $\bG_{\phi_\lambda}=\SO_3^r$ (Proposition \ref{auxgp}) is an adjoint group, \eqref{kernelpi1} must be trivial and thus thecondition \ref{strategy}(a) holds for all $\lambda$.
Since the image of 
\begin{equation}\label{compo}
\bG_\lambda\stackrel{\pi_1^{-1}}{\longrightarrow}\hat\bG_\lambda\stackrel{\pi_2}{\rightarrow}\GL_{3r,\overline E_\lambda}
\end{equation}
is $\SO_3^r\subset\GL_3^r$ (Proposition \ref{auxgp}) and $r$ is the semisimple rank of $\bG_\lambda$, \eqref{compo} can only be 
the adjoint representation on the semisimple part $\mathrm{Lie}(\bG_{\lambda}^{\der})$, i.e., 
the condition \ref{strategy}(b) holds for all $\lambda$.

\subsection{Proof of Corollary \ref{cor1.8}}\label{s3.10}
The compatible system $\{\rho_\lambda\}$ is rectangular by the condition \ref{cor1.8}(c).
Hence, Corollary \ref{cor1.8}(i) follows immediately from Theorem \ref{thm_main2}.
Since $\rho_{\lambda_0}$ is absolutely irreducible (\ref{cor1.8}(c)),
Corollary \ref{cor1.8}(ii) follows directly from Corollary \ref{cor1.8}(i).
Since $\bG_{\lambda,\C}$ is of type $A_1$ for all $\lambda$
(by Corollary \ref{cor1.8}(i) and \ref{cor1.8}(c)), Corollary \ref{cor1.8}(iii) follows directly from Corollary \ref{cor1.8}(ii), 
Proposition \ref{cond}, and Theorem \ref{big}(i).

\section*{Acknowledgments}
We are grateful to Brian Conrad for pointing out two missing rectangular representations of $D_4$, and his comments and suggestions.
C.-Y. Hui would like to thank Kei Yuen Chan, Zachary Feng, Dmitri Whitmore, Kayue Daniel Wong, and Jun Yu for their interests in the article. C.-Y. Hui was partially supported by Hong Kong RGC (no. 17314522), NSFC (no. 12222120), and a Humboldt Research Fellowship. Wonwoong Lee has been supported by the National
Research Foundation of Korea (NRF) grant funded by the Korea government
(MSIP) (No. RS-2024-00341327 and No. RS-2024-00415601 (G-BRL)).

\end{document}